\documentclass[12pt]{amsart} 
\usepackage{amssymb} 
\usepackage{amsfonts} 
\usepackage{amsmath} 
\usepackage{amsthm} 
\usepackage{array} 
\usepackage{multirow}
\usepackage{geometry} 
\usepackage{graphicx,psfrag} 
\usepackage{mathrsfs} 
\usepackage[all]{xy} 
\usepackage{verbatim} 
\usepackage{stmaryrd}
\usepackage{hyperref}
\usepackage{chngpage}
\usepackage{booktabs}
\usepackage[usenames,dvipsnames]{color}
\usepackage{words} 
\theoremstyle{plain} 
\newtheorem{thm}{Theorem}[section] 
\newtheorem{cor}[thm]{Corollary} 
\newtheorem{lem}[thm]{Lemma} 
\newtheorem{lemma}[thm]{Lemma}
\newtheorem{prop}[thm]{Proposition} 
\newtheorem*{thm*}{Theorem}

\newtheorem{defn}[thm]{Definition} 
\theoremstyle{remark} 
\newtheorem*{rem}{Remark} 
\newtheorem*{remark}{Remark}
\newcommand{\Z}{\mathbb{Z}}     
\newcommand{\Q}{\mathbb{Q}} 
\newcommand{\C}{\mathbb{C}} 
\newcommand{\A}{\mathbb{A}}     
\newcommand{\sK}{\mathcal{K}}
\newcommand{\sM}{\mathcal{M}}
\newcommand{\sO}{\mathcal{O}}
\newcommand{\sR}{\mathcal{R}}
\newcommand{\bk}{\textbf{k}}

\newcommand{\cms}{\overline{\sM}} 
\newcommand{\virclass}[1]{{\left[{#1}\right]^\mathrm{vir}}}            
\newcommand{\rma}[2]{\cms_{#1}(#2)}

\newcommand{\abrub}[3]{\cms^\sim_{#1}(#2,#3)}

\newcommand{\Chow}{{\rm CH}}
\newcommand{\br}{{\rm br}}

\providecommand{\sst}{{\rm ss}}
\providecommand{\orb}{{\rm orb}}
\providecommand{\rel}{{\rm rel}}
\newcolumntype{S}{>{\centering\arraybackslash} m{3in} }
\newcolumntype{U}{>{\arraybackslash} m{2in} }
\newcolumntype{T}{>{\centering\arraybackslash} m{1 in} }
\newcommand\T{\rule{0pt}{2.6ex}}	
\newcommand\B{\rule[-1.2ex]{0pt}{0pt}}	

\begin{document}

%
%
\title{Polynomial Families of Tautological Classes on $\sM^{rt}_{g,n}$}
\date{\today}
\author[R. Cavalieri]{Renzo Cavalieri}
\address{Renzo Cavalieri, Colorado State University, Department of Mathematics, Weber Building, Fort Collins, CO 80523, U.S.A}
\email{renzo@math.colostate.edu}
\author[S. Marcus]{Steffen Marcus}
\address{Steffen Marcus, Department of Mathematics, Box 1917, Brown University, Providence, RI, 02912, U.S.A}
\email{ssmarcus@math.brown.edu}
\author[J. Wise]{Jonathan Wise}
\address{Jonathan Wise, Department of Mathematics, Stanford University, Stanford, CA, 94305, U.S.A.}
\email{jonathan@math.stanford.edu}
\begin{abstract}
We study classes $P_{g,T}(\alpha;\beta)$ on $\sM^{rt}_{g,n}$ defined by pushing forward the virtual fundamental classes of spaces of relative stable maps  to an unparameterized $\PP^1$ with prescribed ramification over $0$ and $\infty$.  A comparison with classes $Q_{g,T}$ arising from sections of the universal Jacobian shows the classes $P_{g,T}(\alpha;\beta)$ are polynomial in the parts of the partitions indexing the special ramification data. Virtual localization on moduli spaces of relative stable maps gives sufficient relations to compute the coefficients of these polynomials in various cases.
\end{abstract}
\maketitle
\setcounter{tocdepth}{1}
\setcounter{table}{0}
\tableofcontents
%


\section{Introduction}



We consider the locus $\mathcal{L}$ inside the moduli space $\sM_{g,n}$ of smooth, $n$-pointed, genus $g$ curves over $\C$ consisting of those curves admitting a map to $\PP^1$ with prescribed ramification profile over two points.  This geometric condition can be expressed in two equivalent ways, either as the image of a morphism from an appropriate moduli space of covers of $\PP^1$ (i.e. a Hurwitz space), or by intersecting sections of the universal Jacobian $J_{g,n}$.  Each of these gives an approach to constructing a Chow class corresponding to a closure of $\mathcal{L}$ inside some partial compactification of $\sM_{g,n}$. This paper deals with the comparison and calculation of these two Chow classes in the intersection theory of the moduli space $\sM^{rt}_{g,n}$ of stable, rational tails curves.

In the first approach, ramification data is specified by partitions $\alpha$ and $\beta$ of a positive integer $d$ corresponding to profiles over $0$ and $\infty$ respectively.  Setting $T=l(\alpha)+l(\beta)$, one can define a Chow class in $\cms_{g,T}$ by pushing-forward the virtual fundamental class of the \emph{rubber} or non-rigid (see \cite[Section~2.4]{GV05}) version of the space of relative stable maps \cite{L01,L02} through the stabilization morphism $\mu: \cms^{\sim}_{g}(\PP^1;\alpha 0,\beta \infty)\to \cms_{g,T}$.  We call this class $P_{g,T}(\alpha,\beta)$ (see Definition \ref{def:pclass}).  These classes are introduced by Graber and Vakil in \cite[p.~22]{GV05}, and shown to be tautological in \cite{FP05}.  Alternatively, given a tuple $(k_1, \ldots, k_T)$ of integers adding to $0$, ramification data may be specified by the condition that the divisor $\sum k_ip_i$ is principal.  In this case, $\mathcal{L}$ consists of the inverse image of the zero section $Z\subset J_{g,T}$ through the section $\sigma:(C;p_1,\ldots, p_T)\mapsto \sum_i k_ip_i$ of the universal Jacobian.  One can naturally extend $\sigma$ to to a section of $J_{g,T}$ over the moduli space of curves of compact type and, by pulling back $Z$, define a Chow class $Q_{g,T}(k_1,\ldots,k_T) \in R^g(\sM^{ct}_{g,T})$ (see Definition \ref{def:qclass}).  An argument due to Ravi Vakil (Theorem \ref{thm:Qpoly}) shows that the $Q_{g,T}$ classes are Chow valued polynomials in the variables $(k_1, \ldots, k_T)$. 

The smallest system of partial compactifications of moduli spaces of smooth curves closed under pullback via forgetful morphisms is $\sM^{rt}_{g,T}$ (curves with rational tails), parameterizing stable curves with one irreducible component of geometric genus $g$. The goals of this paper are to provide a comparison between $P_{g,T}$ and $Q_{g,T}$ over $\sM_{g,T}^{rt}$ and to compute $P_{1,T}$ in terms of standard tautological classes using this comparison and localization techniques.  Recent work of Richard Hain provides a remarkable and completely general explicit computation of the class $Q_{g,T}$ in $H^{2g}(\sM_{g,T}^{ct})$ \cite{H11}.  When $g\geq2$, Hain's formula  takes the form:
\begin{thm}\cite[Theorem~11.1]{H11}
\begin{align*}
Q_{g,T}(k_1,\ldots,k_T) = \frac{1}{g!}\left(\sum_{j=1}^{n}\frac{k_j^2\psi_j^\dagger}{2}-\sum_{\substack{J\subset[n]\\ |J|\geq2}}\sum_{i,j\in J}k_ik_j\delta_0^J-\frac{1}{2}\sum_{J\subset[n]}\sum_{h=1}^{g-1}\left(\frac{2h-1}{2g-2}\sum_{j\in J}k_j\right)^2\delta_h^J\right)^g.
\end{align*}
\end{thm}  
\noindent The notation $\delta_h^J$ refers to the boundary divisor $\overline{\text{D}}_{h,0}(J|[T]-J)$, and the psi classes $\psi_j^\dagger$ are pulled back from $\cms_{g,1}$.  A similar formula holds in the genus 1 case \cite[Theorem~12.1]{H11}.  

\subsection{Statement of Theorems}

Our first theorem establishes a comparison of the restriction of these classes to $\sM_{g,T}^{rt}$. Since we are only concerned with rational tails, we abuse notation and write $P_{g,T}$ and $Q_{g,T}$ for their restrictions. 
\begin{thm}\label{thm:compare}
Restricting to $\sM^{rt}_{g,T}$, we have: $P_{g,T}(\alpha,\beta) = Q_{g,T}(\alpha,-\beta) \in R^g(\sM_{g,T}^{rt})$.
\end{thm}


Having established the equality of the $P$ and $Q$ classes, we next use, independently of Theorem 1.1, the Gromov-Witten theoretic tool of virtual localization to understand the coefficients of these polynomials.  First we recover a variant of classical result of Looijenga's \cite[Lemma~2.10]{l:trmg}, presented as it appears in \cite{GJV06}.

\begin{thm}[\cite{GJV06} Theorem 3.5]
\label{theorem:T=2}
\[
P_{g,2}(d;d)= d^{2g} P_{g,2}(1;1)
\]
where
$$
\sum_{g=1}^\infty\lambda_g\lambda_{g-1}P_{g,2}(d;d)y^{2g}=\log\left(\frac{dy/2}{\sin(dy/2)}\right)[pt.]
$$
\end{thm}

\noindent Another localization computation determines explicitly the polynomial for $T=3$ and $g=1$.

\begin{thm}
\label{theorem:T=3}
The genus 1 total length 3 polynomial $P_{1,3}(d;\alpha_2,\alpha_3)=A_2\alpha_2^2+A_3\alpha_3^2 +B\alpha_2\alpha_3$ has coefficients
\begin{align*}
A_2&=\psi_1-\overline{\text{D}}_{1,0}(2|1,3)\\
A_3&=\psi_1-\overline{\text{D}}_{1,0}(3|1,2)\\
B&=\psi_1-\overline{\text{D}}_{1,0}(1|2,3)
\end{align*}
in  $\sR^1(\sM_{1,3}^{rt})$.
\end{thm}
By Corollary \ref{cor:pullback}, this computes the polynomial in genus 1 for arbitrary total length.   

\begin{cor}\label{cor:arbitraryT} Let $[T]:=\{1,\ldots,T\}$. 
For any $T\geq 3$: 
$$P_{1,T}(d;\alpha_2, \ldots, \alpha_T)=\sum_{i=2}^T A_i\alpha_i^2+ \sum_{i>j} B_{i,j}\alpha_i\alpha_j$$
where:
\begin{align*}
A_i&=\psi_1-\sum_{\substack{J\subset[T]:\\ |J|\geq2\\ 1\in J \text{ and } i\notin J}}\overline{\text{D}}_{1,0}([T]-J|J)\\
B_{i,j}&=\psi_1-\sum_{\substack{J\subset[T]:\\ |J|\geq2\\ 1\in J \text{ and } i,j\notin J}}\overline{\text{D}}_{1,0}([T]-J|J)-\sum_{\substack{J\subset[T]:\\ |J|\geq2\\ 1\notin J \text{ and } i,j\in J}}\overline{\text{D}}_{1,0}([T]-J|J)\\
\end{align*}
\end{cor}  

\subsection{Comparison of Virtual Classes}  We prove Theorem~\ref{thm:compare} using a theorem of Costello \cite[Theorem~5.0.1]{C06}, which reduces the problem to two parts:  a comparison of the obstruction theories defining the virtual classes, and a verification of the statement of the theorem in a ``universal situation''.  This ``universal situation'' is described in Section~\ref{sec:compare}, where it is denoted $\oM_{\rel}(\sP)$.  It is the universal example of relative stable maps to expansions of a space with two disjoint marked sections.  The universal version of Theorem~\ref{thm:compare} amounts to the statement that $\oM_{\rel}(\sP)$ is birational to the moduli space of rational tails curves; this is proved in Proposition~\ref{DMtype}.

The obstruction theory comparison is more technical.  It requires an explicit understanding of J.\ Li's definition of the virtual fundamental class for the moduli space of relative stable maps \cite{L02}.  While the definition of the virtual fundamental class for the moduli space of relative stable maps to ``rubber'' $\bP^1$ is analogous to one defined by Li (as was pointed out in \cite{GV03}), the details of the construction of the virtual class have not appeared before.

In order to provide these details, we reinterpret J.\ Li's construction in the present context in terms of a site on which deforming a relative stable map becomes a locally trivial problem.  This site is introduced in Section~\ref{sec:site}, where we discuss some of its basic properties.  In a ``locally trivial'' situation, there is always a natural obstruction theory arising from torsors under a relative tangent bundle.  (This principle will be explained in greater detail in the forthcoming~\cite{obs}, where it will also be applied to deformation problems including the one originally studied by Li.)  This idea yields a geometric description of the relative obstruction theory for $\sM_g^\sim(\alpha,\beta)^{rt}$ relative to the ``universal situation'' $\oM_{\rel}(\sP)$.  This obstruction theory is visibly the same as the pullback of the normal bundle of the zero section of the relative Jacobian, which is by definition the obstruction theory that defines the $Q_{g,T}$.
   
\subsection{Polynomial classes}  Besides being interesting on its own, Theorem \ref{thm:compare} is essential in establishing the polynomiality of the classes $P_{g,T}$.
This is not the first time that ``polynomiality behavior'' of tautological classes related to stable maps to $\PP^1$ has appeared. Perhaps the most famous example is given by the ELSV formula (\cite{ELSV1},\cite[Theorem 1.1]{ELSV2})
expressing simple Hurwitz numbers $H^g_\alpha$ in terms of Hodge integrals. This numerical statement is obtained by integrating the tautological class
$$
[\sC_\alpha]=\mu_\ast(\virclass{\rma{g}{\alpha}}\cap \text{br}^\ast[pt]).
$$  

\noindent This class parameterizes maps to $\PP^1$ with profile $\alpha$ over zero and a fixed generic branch divisor, thus   $ [\sC_\alpha] = H^g_\alpha [pt]$. On the other hand after choosing an appropriate equivariant lift of $[\sC_\alpha]$ one can evaluate it using the Atiyah-Bott localization isomorphism:
$$
[\sC_\alpha]= r^g_\alpha!\prod_{i=1}^{l(\alpha)}\left(\frac{\alpha_i^{\alpha_i}}{\alpha!}\right)   \frac{1-\lambda_1+\cdots\pm\lambda_g}{(1-\alpha_1\psi_1)\cdots(1-\alpha_{l(\alpha)}\psi_{l(\alpha)})}
 \in \sR^{3g-3+l(\alpha)}(\cms_{g,l(\alpha)}),
$$
where this expression is understood by expanding the denominator terms into geometric series and considering products of terms of degree $3g-3+l(\alpha)$. It is immediate to conclude that, other than a combinatorial prefactor, $[\sC_\alpha]$ is polynomial in the $\alpha_i$'s with coefficients given by monomials in $\psi$ classes and one Hodge class. 


The above proof of the ELSV formula (from \cite[Section 5]{GV03}) motivates our approach to Theorems \ref{theorem:T=2} and \ref{theorem:T=3}. However, two significant obstacles arise: first, the classes we study are ``rubber" classes and live in moduli spaces that do not admit a torus action; second, having two relative points for our maps to $\PP^1$ already ``crowds'' both $0$ and $\infty$, leaving no fixed point for curves to contract to, and therefore no fixed locus containing moduli spaces of curves as one of the factors.

We localize on moduli spaces of relative maps to $\PP^1$ with only one relative point. Fixed loci consist of products of moduli spaces of curves and moduli spaces of rubber stable maps with two relative points. Hence the classes we are interested in appear in the fixed loci. By choosing carefully the auxiliary integrals, we produce manageable relations between rubber classes and standard classes. Remebering the polynomiality of rubber classes, each such relation translates into a linear equation in the coefficients of $P_{g,T}$, giving a linear system of equations with solution expressed in terms of standard classes.
  
\subsection{On the problem of extending the class $Q_{g,T}$} Theorems \ref{theorem:T=2} and \ref{theorem:T=3} are special cases of the more general result of Richard Hain's cited above. Hain's techniques  are extremely well tuned to the study of the classes $Q_{g,T}$ and give general formulas that are beyond our reach. However, we now hope that the two points of view may interact. A natural question posed by Hain is how to further extend the classes $Q_{g,T}$.  Relative stable maps, for example, provide a way of extending these classes to the compact moduli space. There are, in fact, more ways to compactify $Q_{g,T}$, such as using moduli spaces of admissible covers. It would be interesting to understand the relationship between the various possible compactifications, and to see if one of them is particularly natural. To make this vague statement only infinitesimally more precise, consider the following diagram:

$$
\xymatrix{   & \overline{J}_{g,T} \ar[d] \ar[dl] & \\
\widehat{\sM}_{g,T} \ar@/^/[ur]^{\hat{\sigma}} \ar[r]^p & \overline{\sM}_{g,T} \ar@/_/@{.>}[u]_{\sigma} & \sM_{g,T}^{ct} \ar[l] \ar@/_/[ul]_{\sigma}.
}
$$
Here $\overline{J}_{g,T}$ is meant to be ``some" compactified universal Jacobian, and $\widehat{\sM}_{g,T}$ ``some" space allowing a resolution of the indeterminacies of the section $\sigma$ used to define $Q_{g,T}(k_1,\ldots, k_T)= \sigma^\ast(Z)$. A natural compactification of $Q_{g,T}(k_1,\ldots, k_T)$ would be a class of the form:
$$
\overline{Q}_{g,T}(k_1,\ldots, k_T)=p_\ast \hat{\sigma}^\ast (Z)
$$
 i.e. a class obtained by resolving the indeterminacies of the section of the Jacobian (requiring us to work on a ``larger" space than $\cms_{g,T}$), pulling back the zero section, and pushing forward again to $\cms_{g,T}$.
\subsection{Acknowledgments}
The authors would like to thank Ravi Vakil for suggesting this project to us, and for providing the argument outlined in Section 2.  We are grateful to Dan Abramovich for his numerous insights, including his suggestion of the cartesian diagram to use for the comparison and the argument of Proposition \ref{DMtype}.  We would also like to thank Richard Hain for speaking to us about his related work.  This project has benefited from many useful conversations with Brian Conrad, Dan Erman, Barbara Fantechi, William Gillam, Jack Hall, Richard Kenyon, Kelli Talaska, and Kevin Tucker. The second author is supported by an NSERC PGS-D grant and the NSF award~0603284.  The third author is partly supported by NSF MS-PRF~0802951.  Parts of this work were accomplished at MSRI during the Algebraic Geometry program in the Spring of 2009.

%
\section{Polynomiality of $Q_{g,T}$}
\label{sec:polynomiality}

In this section we introduce a family of tautological classes $Q_{g,T}(k_1,\ldots,k_T)\in \Chow^g(\sM^{rt}_{g,T})$ constructed from the geometry of the universal Jacobian. We present an argument, due to Ravi Vakil, showing that the function $Q_{g,T}(k_1,\ldots,k_T)\in \Chow^g(\sM^{rt}_{g,T})$ is a homogeneous Chow valued polynomial in the variables $k_1,\ldots,k_T$. 

Let $\Z_0^T$ denote $T$-tuples of integers summing to zero.  We write $Z$ for the zero section of the universal Jacobian $\rho: J_{g,T} \rightarrow \sM_{g,T}^{rt}$ over the rational tails locus.  Given a rational tails curve $C$, let $\pi:C\to \oC$ be the contraction to the unique smooth genus $g$ component. 

For $(k_1, \dots, k_T) \in \Z^T_0$, define the section 
\[
\sigma_{(k_1,\ldots,k_T)}: \sM^{rt}_{g,T} \rightarrow J_{g,T}
\] by
\[
(C, p_1, \dots, p_T) \mapsto \left( C, \pi^\ast\sO_{\oC} \left(  k_1 \pi(p_1) + \cdots k_T 
  \pi(p_T)  \right) \right).
\]
For each $g \geq 0$, $T \geq 2$ the intersection of $\sigma_{(k_1,\ldots,k_T)}$ with the zero section determines a Chow-valued function
\begin{align*}
Q_{g,T}: \Z^T_0 &\rightarrow \Chow^g(\sM_{g,T}^{rt})
\end{align*}
defined as follows:
\begin{defn}\label{def:qclass} The class $Q_{g,T}(k_1,\ldots,k_T)$ is the $g$-codimensional Chow class
\begin{align*}
Q_{g,T}(k_1,\ldots,k_T) &:= \sigma_{(k_1,\ldots,k_T)}^* [Z]\in \Chow^g(\sM_{g,T}^{rt}).
\end{align*}
\end{defn}

\begin{thm}\label{thm:Qpoly}\text{}
\begin{enumerate}
\item[(i)] The function $Q_{g,T}$ is equivariant in its $T$ entries with respect to the action of the symmetric group permuting the marked points.
\item[(ii)]  If $F: \sM_{g,T}^{rt} \rightarrow \sM_{g,T-1}^{rt}$ is the
  forgetful morphism,
\[
Q_{g,T}(x_1, \dots, x_{T-1}, 0) = F^* Q_{g, T-1}(x_1, \dots,
x_{T-1}).
\]
\item[(iii)] $Q_{g,T}$ is a (Chow-valued) polynomial.
\item[(iv)] $Q_{g,T}$ is homogeneous of degree $2g$.
\end{enumerate}
\end{thm}

Notice first that (ii) implies the following:

\begin{cor} \label{cor:pullback} For fixed $g$, all $Q_{g,T}$  are determined by the Chow-valued polynomial
$Q_{g,2g+1}$.
\end{cor}

Theorem \ref{thm:Qpoly} relies on the following result of Deninger and Murre, extending work of Beauville and, earlier, Mukai.

\begin{thm}[{\cite[Thm.~2.19]{DM91}}]\label{thm:Ftransform} Suppose $\pi: A \rightarrow X$ is an abelian scheme over a smooth finite type stack $X$.  For any $k \in \ZZ$ let $\bk : A \rightarrow A$ be the multiplication-by-$k$ map.  For each nonnegative integer $t$, define 
\[
\Chow_{p,t}(A) = \{ \xi \in \Chow_p(A) : \bk^* \xi = k^t \xi \text{ for all $k \in \Z$} \}
\]
Then $\Chow_p(A) = \mathop{\bigoplus}_{t = 0}^N \Chow_{p,t}(A)$ where $N$ depends only on $dim A/X$ and $\dim X$.
\end{thm}
\begin{proof}[Proof of Theorem~\ref{thm:Qpoly}]
The function $Q_{g,T}$ satisfies (i) by construction, and (ii) follows from the commutativity of
\[
\xymatrix{
J_{g,T} \ar[r] \ar[d] & \sM_{g,T}^{rt} \ar[d] \\
J_{g,T-1} \ar[r] & \sM_{g,T-1}^{rt}. }
\]

To prove (iii) and (iv), we reinterpret the definition of $Q_{g,T}(k_1, \dots, k_t)$ as follows.  Consider the diagram
\[
\xymatrix{J^{T-1}_{g,T} \ar[rr]^-{(\bk_1, \dots, \bk_{T-1})} \ar[drr] & & J_{g,T}^{T-1} \ar[rr]^\Sigma \ar[d]  & & J_{g,T} \ar[dll] \\
& & \sM_{g,T}^{rt} \ar@/^/[ull]^{\tau} \ar@/_/[urr]_{\sigma_{k_1, \dots, k_t}}}
\]
where
\begin{itemize}
\item  $J^n_{g,T} $ is the $n$-th fiber power of $J_{g,T}$ over the base $\sM_{g,T}^{rt}$,
\item  $\tau$ is the section $(C, p_1, \dots, p_T) \mapsto  (C, \sO_C(p_1- p_T), \sO_C(p_2-p_T), \dots, \sO_C(p_{T-1}-p_T)),$
\item  $(\bk_1, \dots, \bk_{T-1})$ denotes factor-wise  multiplication in the abelian scheme, and
\item  and $\Sigma$ is summation. 
\end{itemize}   
Then  $\sigma_{k_1, \dots, k_T} = \Sigma \circ ( \bk_1,
\dots, \bk_{T-1}) \circ \tau$, so 
\[
Q_{g,T}(\bk_1, \dots, \bk_T) = \tau^* \: ( \bk_1, \dots, \bk_{T-1})^* \:
\Sigma^* [Z] .
\]

\noindent The morphism $(\bk_1, \dots, \bk_{T-1})$ factors as
\[
(\bk_1, 1, \dots, 1) \circ (1, \bk_2, 1, \cdots) \circ \cdots \circ
(1, \dots, 1, \bk_{T-1}).
\] 

Applying Theorem \ref{thm:Ftransform} to each $i \in \{1, \dots, T-1\}$ (taking $A = J^T$ and $X = J^{T-1}$) shows that $Q_{g,T}(\bk_1, \dots, \bk_T)$ is a polynomial in each $\bk_i$ separately, of degree at most $N$.  This implies that $Q_{g,T}$ is a polynomial by Lemma \ref{lem:varpoly}, below, and establishes (iii).

Finally, to prove (iv), note that
$$
Q(t \bk_1, \dots, t \bk_T) = \tau^* \: ( \bk_1, \dots, \bk_{T-1})^* \:
\Sigma^* \: \mathbf{t}^* [Z],$$
and $\mathbf{t}^*[ Z] = t^{2g} [Z]$ by Proposition \cite[Proposition~2.18]{DM91}, as observed in  \cite[Thm.~2.10]{l:trmg}.  Thus $Q(t k_1, \cdots, t k_T) = t^{2g} Q(k_1, \dots, k_T)$.
\end{proof}

\begin{lem} \label{lem:varpoly}
Suppose  $Q(x_1, \dots, x_t)$ is a function from $(\Z_{\geq 0})^t$ to a $\Q$-vector space, such if all but one of variables are fixed, the resulting single-variable function is a polynomial of degree at most $N$.  Then $Q$ is a polynomial. 
\end{lem}

\begin{proof}
The function $Q$ is determined by its values on $\{ 0, \dots, N \}^t$ using the interpolation formula for polynomials.  The interpolation formula describes $Q$ as a polynomial.
\end{proof}

\begin{rem}
The above arguments above hold without change over the locus $\sM_{g,T}^{ct}$ of curves of compact type. In homology, many of these results follow directly from the explicit computations of Richard Hain \cite{H11}. 
\end{rem}

%
%
\section{Comparison of virtual classes} \label{sec:compare}

Denote by $$\cms_{g,n}(\alpha,\beta) := \cms_{g,n}(\PP^1;\alpha 0,\beta \infty)$$ the moduli space of degree $d$ stable maps to $\PP^1$ relative to the points $0$ and $\infty$ with prescribed ramification given by partitions $\alpha\vdash d$ and $\beta\vdash d$ respectively.  Let $l(\alpha)$ and $l(\beta)$ be the lengths of the partitions, and $T=l(\alpha)+l(\beta)$ the total length.  We denote by $\abrub{g,n}{\alpha}{\beta}$ the variant of this space in which the target is an unparameterized or ``rubber" $\PP^1$.  In Theorem 1 of \cite{FP05}, Faber and Pandharipande show that the pushforward through the forgetful stabilization morphism $\mu$ of the respective virtual classes lie in the tautological ring of $\cms_{g,n+T}$.

We are concerned with the pushforward $\mu_\ast\virclass{\abrub{g}{\alpha}{\beta}}\in\sR^\ast(\cms_{g,T})$ and, in particular, its restriction to rational tails. 

\begin{defn}\label{def:pclass} The tautological class $P_{g,T}(\alpha,\beta)$ is the $g$-codimensional Chow class
\[
P_g(\alpha_1,\ldots,\alpha_{l(\alpha)};\beta_1,\ldots,\beta_{l(\beta)}):=\mu_\ast\virclass{\abrub{g}{\alpha}{\beta}^{rt}} \in \sR^g(\sM^{rt}_{g,T}).
\]
\end{defn}

In this section and the next, we will simplify the notation for our moduli spaces by suppressing the various subscripts for locally constant data.  The reader may imagine either that these data have been fixed, or else that each moduli space is the disjoint union over discrete parameters of moduli spaces with appropriate decorations.  We shall write:

\begin{itemize}
  \item $\oM_{\rel}(\sP/\BGm)$ for the moduli space of stable relative maps from curves with rational tails to ``rubber~$\bP^1$'';
  \item $\oM_{\rel}(\sP)$ for the moduli space of stable relative maps from curves with rational tails to to $\sP = [\bP^1 / \Gm]$;
  \item $J$ for the relative Jacobian over the moduli space of smooth curves;
  \item $Z$ for the moduli space of smooth curves, embedded as the zero section of its relative Jacobian.
\end{itemize}

We prove Theorem~\ref{thm:compare} using a comparison technique introduced by Costello \cite{C06}.  We will show in Theorem~\ref{thm:cart} that the square
\begin{equation} \label{diag:costello}\xymatrix{
      \oM_{\rel}(\sP / \BGm) \ar[r] \ar[d] & Z \ar[d] \\
      \oM_{\rel}(\sP) \ar[r] & J
    }
\end{equation}
is cartesian.  This provides a second relative obstruction theory for $\oM_{\rel}(\sP/\BGm)$ over $\oM_{\rel}(\sP)$, obtained by pullback from the normal bundle of $Z$ in $J$, in addition to the natural one that it used to define the virtual fundamental class of $\oM_{\rel}(\sP/\BGm)$.

In Sections~\ref{sec:mapsintofibers} and~\ref{sec:cartesian} we describe the stacks $\oM_{\rm rel}(\sP / \BGm)$ and $\oM_{\rm rel}(\sP)$, explain the arrows in (\ref{diag:costello}), and show that this diagram is cartesian.  In Section~\ref{sec:costello} we show the diagram satisfies the further hypotheses of Costello's Theorem \cite[Theorem~5.0.1]{C06}, and provide a proof of our theorem, contingent on some obstruction theoretic details that we postpone to Section~\ref{sec:OT}.

%
\subsection{Moduli stacks of targets}
Let $\sT$ be Jun Li's stack of expanded degenerations, parameterizing expansions of the target for relative stable maps.  For more background about $\sT$, see \cite[Definition~4.4]{L01} (where $\sT$ is denoted $\fZ^{\rm rel}$), \cite[Section~2.5]{GV05}, and \cite{ACFW}.  It was pointed out in~\cite{GV05} that $\sT$ is the moduli space of $3$-marked curves such that all nodes separate the first two markings from the third; a proof of this appears in~\cite{ACFW}.

Using this interpretation, we can identify $\sT^2 = \sT \times \sT$ with the moduli space of $3$-marked semi-stable curves such that all nodes separate the first and second markings.  One may construct such a family over $\sT \times \sT$ by gluing together the two families corresponding to the first and second projections along the components containing the first two markings, which may be identified canonically with $\bP^1$.  The reader may verify that this map gives the claimed isomorphism, say by constructing an inverse.

The stack $\sT^2$ will play the role of the moduli space of targets for relative stable maps to~$\sP$.

\begin{rem}
The stack $\sT^2$ can be also viewed as an open substack of the universal family of the moduli space of two pointed semi-stable rational curves $\fM_{0,2}^{\sst}$.
\end{rem}

The space of targets for relative maps to a non-rigid target is denoted $\sT_\sim$ in \cite{GV05}.  It is the open substack of $\fM_{0,2}$ parameterizing chains of rational curves where all nodes separate the two marked points.  Since we are working with rational tails curves, we have the privelege of working with a slightly different moduli space of targets for non-rigid relative stable maps.

Let $C$ be the source of a relative stable map and assume that $C$ is a rational tails curve with $C \rightarrow \oC$ the contraction of $C$ onto its distinguished irreducible component.  The image of $\oC$ in the target expansion of $\sP$ distinguishes a specific component of the expanded target.  Therefore, all of the rubber stable maps considered here will come with a distinguished component of the expanded target, and we build this datum into our definition of the moduli space of expanded targets.

To be precise, we define a stack $\tsT$ of expanded targets for rubber maps.  An $S$-point of $\tsT$ is a family of $2$-marked semistable curves $P\to S$ together with a $\Gm$-equivariant embedding of a $\Gm$-torsor $Q$ over $S$ into $P$ (recall that $P$ has a canonical balanced action of $\Gm$ \cite[Proposition~3.2.1]{ACFW}).  At the level of $\C$-points, the torsor $Q$ is simply the smooth $\C^\ast$ given by the complement of the two nodes inside the distinguished rational component of the chain.

There is a map $\sT^2 \rightarrow \tsT$ forgetting the third marking and the parameterization of the distinguished component.

\begin{prop} \label{prop:1}
  The natural map $\sT^2 \rightarrow \tsT$ admits a retraction, inducing an isomorphism $\tsT \simeq \sT^2 \times \BGm$.
\end{prop}
\begin{proof}
  Let $P \rightarrow S$ be an $S$-point of $\tsT$.  By definition, we are given a $\Gm$-torsor $Q$ over $S$ and a $\Gm$-equivariant embedding $Q \rightarrow P$.  Let $P' = P \tensor Q^\vee$, where $Q^\vee$ is the opposite torsor of $Q$  and $P \tensor Q^\vee = P \fp^{\Gm} Q^\vee$, the quotient of $P \fp_S Q$ by $\Gm$, acting diagonally (equivalently, the quotient of $P \fp_S Q^\vee$ by $\Gm$ acting anti-diagonally).  We have an equivariant embedding $\Gm = Q \tensor Q^\vee \rightarrow P \tensor Q^\vee$.  The image of the identity section of $\Gm$ gives a section of $P'$, so $P'$ is an object of $\sT^2$.  The pair $(P', Q)$ gives a map $\tsT \rightarrow \sT^2 \times \BGm$.  Since $P = P' \tensor Q$, we can recover $P$ uniquely from $P'$ and $Q$ so this map is an isomorphism.
\end{proof}

%
\subsection{Stable maps into the fibers of $\sP\to B\Gm$ and $\sP\to pt$} \label{sec:mapsintofibers}
 Denote by $\sP \simeq [\bP^1 / \Gm]$ the universal $\PP^1$ bundle over $B\Gm$.  An $S$-point of $\sP$ is a tuple $(U, V, L, z, w)$ where $U$ and $V$ are open subsets of $S$ that together form a cover, $L$ is a line bundle on $S$, and $z \in \Gamma(U, L)$ and $w \in \Gamma(V, L^\vee)$ are sections such that $z \rest{U \cap V} w \rest{U \cap V} = 1$.  The images in $\sP$ of $0$ and $\infty$ from $\bP^1$ are divisors, denoted $\sD_+$ and $\sD_-$ respectively, with $\sD_+$ equal to the vanishing locus of $z$ (which is contained inside $U$) and $\sD_-$ equal to the vanishing locus of $w$ (and contained inside $V$).  The line bundle $L$ gives the projection $\sP \rightarrow \BGm$.

An $S$-point of the stack $\abrub{g}{\alpha}{\beta}^{rt}$ is given by a commutative diagram
\begin{equation} \label{eqn:2} \xymatrix{
    C \ar[r] \ar[d] & \tsP \ar[d] \\
    S \ar[r] & \tsT 
  }
\end{equation}
where $\tsP$ is the universal curve over $\tsT$.  The family $C / S$ is a family of rational tails curves, the diagram is predeformable with finite automorphism group, and the order of contact of $C$ along the marked sections of $\tsP$ over $\tsT$ coincides with the partitions $\alpha$ and $\beta$.  The fibers of $C\to S$ have a marking for each of the parts of the partitions $\alpha$ and $\beta$, corresponding to the points of $C$ in the pre-images of $0$ and $\infty$ respectively.  The map $C\to \tsP$ restricts on fibers to relative stable maps to a rubber $\PP^1$.    This description of $\abrub{g}{\alpha}{\beta}^{rt}$ gives an identification 
\[
\coprod_{g,\alpha,\beta} \abrub{g}{\alpha}{\beta}^{rt} = \oM_{\rm rel}(\sP / \BGm)
\]  
of the disjoint union over the discrete data with the moduli space of relative stable maps from rational tails curves into the fibers of $\sP\to B\Gm$ (as considered in \cite{AF,ACW10}).

Let $\oM_{\rel}(\sP)$ be the moduli stack of commutative diagrams
\begin{equation} \label{eqn:1} \xymatrix{
    C \ar[r] \ar[d] & \tsP \ar[d] \\
    S \ar[r] & \sT^2
  }
\end{equation}
that are  predeformable as above, with contact order and marked points again determined by a choice of partitions $\alpha$ and $\beta$.  A family $C\to S$ of rational tails curves comes with a distinguished genus $g$ component in the fibers.  Contracting the rational tails in the fibers determines a map $\pi:C\to \oC$ to a family of smooth genus $g$ curves.  

%
\subsection{The Costello diagram} \label{sec:cartesian}
It is enough for us to work with the relative Jacobian $J=\coprod J_{g,T}$ over the moduli space of smooth curves.  Again, let $Z$ be the zero section.

\begin{thm}\label{thm:cart}
  There is a cartesian square
  \begin{equation}\label{diag:costelloagain} \xymatrix{
      \oM_{\rm rel}(\sP / \BGm) \ar[r] \ar[d] & Z \ar[d] \\
      \oM_{\rm rel}(\sP) \ar[r] & J .
    }
  \end{equation}
\end{thm}

This is Diagram (\ref{diag:costello}).  The top horizontal arrow is given by sending a square
\[
 \xymatrix{
    C \ar[r] \ar[d] & \tsP \ar[d] \\
    S \ar[r] & \tsT }
\]
to the stabilization $\oC$ of the marked curve $C$ (recall that $Z \cong \sM_{g,T}$).  Denote by $\{p_i\}$ and $\{q_j\}$ the marked points on $C$ determined by the parts of $\alpha$ and $\beta$ respectively.  The bottom arrow sends a diagram~\eqref{eqn:1} to the pair 
 \[
\displaystyle \biggl(\oC, \cO_{\oC}\Bigl(\sum_i \alpha_i \pi(p_i)-\sum_j \beta_j\pi(q_j)\Bigr)\biggr)
 \]  
where $\oC$ again denotes the stabilization of the marked curve $C$.  This is just the image through the section $\sigma_{(\alpha,-\beta)}$ of Section~\ref{sec:polynomiality}.  The left vertical arrow is given by composition with $\tsT \rightarrow \sT^2$.  The right vertical one is the obvious inclusion.  We now show the diagram in question is cartesian.\\

\noindent\emph{Proof of Theorem \ref{thm:cart}}.

\textbf{The diagram is commutative.}  The composition of the left vertical arrow and the lower horizonatal arrow gives $\oC$ with the line bundle
\[
M := \cO_{\oC}\Bigl(\sum_i \alpha_i \pi(p_i)-\sum_j \beta_j\pi(q_j)\Bigr).
\] 
The condition that this object lie in $Z$ is that $M$ be pulled back from $S$.  To demonstrate this, we will actually identify which line bundle $M$ pulls back from.

There is a commutative diagram
\begin{equation*} \xymatrix{
    C \ar[r] \ar[d] & \tsP \ar[r] \ar[d] & \sP \ar[d] \\
    S \ar[r] & \tsT \ar[r] & \BGm .
  }
\end{equation*}
The map $S \rightarrow \BGm$ gives a line bundle $\oQ$ on $S$ (the completion of the torsor $Q$ associated to the map $S \rightarrow \tsT$), which pulls back to the line bundle $L$ on $C$ associated to the map $C \rightarrow \sP$.  To show that $M$ is pulled back from $S$, it is enough to show that $\pi^\ast M$ is isomorphic to $L$.  This is because for any line bundle $F$ on $\oC$ we have a canonical isomorphism $\pi_\ast \pi^\ast F = F$.

We prove that $\pi^\ast M \cong L$ by a deformation theory argument.  If $S$ is a point, it is obvious, since both are line bundles on $C$ that have degree zero on every component and they agree on the central component.  Assume now it is true over some infinitesimal extension $S$ of a point $S_0$, and that $S'$ is a square-zero extension of $S$ by $\cO_{S_0}$ (and we have compatible data~\eqref{eqn:2} appropriately decorated).  Then it is true over $S'$, since isomorphism classes of deformations of the line bundles $M$ and $L$ are classified by $H^1(\oC, \cO_{\oC})$ and $H^1(C, \cO_C)$ respectively.  By \cite[Lemma~3.1.7]{ACW10}, these are isomorphic via the natural map, so an isomorphism between $\pi^\ast M$ and $L$ can be extended to an isomorphism between $\pi^\ast M'$ and $L'$.

This proves that $\pi^\ast M$ and $L$ agree in a formal neighborhood of every point of a general $S$.  To show they agree on all of $S$ we can assume, since the moduli problem is locally of finite presentation, that $S$ is Noetherian.  Then by Grothendieck's existence theorem, $\pi^\ast M$ and $L$ must agree on the formal completion of $S$ at any point.  This implies that the locus where the two line bundles agree is stable under generization.  On the other hand, the locus where they agree is the pullback of the zero locus of the relative Jacobian, thus is also closed.  Since this locus also includes all of the points of $S$, it must be $S$ itself.

\textbf{The diagram is cartesian.}  An object of the fiber product $\oM_{\rel}(\sP)\fp_J Z$ consists of a diagram \eqref{eqn:1}  such that the line bundle $L$ on $\oC$ (defined above) is pulled back from $S$.  To lift this to a point of $\oM_{\rel}(\sP / \BGm)$, we need to factor the map $S \rightarrow \sT^2$ through $\tsT = \sT^2 \times \BGm$ so that Diagram~\eqref{eqn:2} commutes.  This means we have to find a map $S \rightarrow \BGm$ so that the compositions $C \rightarrow \tsP \rightarrow \sP \rightarrow \BGm$ and $C \rightarrow S \rightarrow \tsT \rightarrow \BGm$ agree.  But the first of these is $\pi^\ast M$ and the second is $L$, which we just saw are isomorphic.  Since $L$ is pulled back from $S$, so is $M$.  \qed

%
\subsection{Proof of the comparison}\label{sec:costello}

To prove our comparison theorem, we reduce the problem to an application of \cite[Thm.~5.0.1]{C06}.  Following our proof, the rest of this section and Section \ref{sec:OT} are devoted to ensuring the relevant hypotheses are met. 

\begin{proof}[Proof of Theorem \ref{thm:compare}.]
  Denote by $\fM=\coprod_{g,T} \fM_{g,T}$ the Artin stack of Deligne-Mumford pre-stable curves.  Let $\fM^\ast$ be the stack of \emph{stable} rational tails curves with disjoint marked points, weighted by integers $k_1, \ldots, k_T$, such that $\sum k_i = 0$.  The map $\sigma$ from Section \ref{sec:polynomiality} determines a map $\oM_{\rm rel}(\sP)\to\fM^\ast$.  The bottom horizontal map in (\ref{diag:costelloagain}) factors through $\fM^{\ast}$, giving a diagram
  \[
  \xymatrix{
    \oM_{\rm rel}(\sP / \BGm) \ar[r]^-\mu \ar[d] & \fZ \ar[r] \ar[d]& Z \ar[d] \\
    \oM_{\rm rel}(\sP) \ar[r] & \fM^{\ast} \ar[r] & J .
  }
  \]
  where $\fZ = Z \fp_J \fM^{\ast}$ is the pullback of $Z$ to $\fM^{\ast}$ and both squares are cartesian.  We equip $\fZ$ with the relative obstruction theory pulled back from that of $Z$ over $J$.    

  We will show in Corollary~\ref{cor:absolutevirclass} that the absolute virtual class for $\oM_{\rel}(\sP/\BGm)$ coincides with the virtual class relative to $\oM_{\rel}(\sP)$:
  \begin{equation*}
   \Big[ \oM_{\rel}(\sP/\BGm) \Big/ \oM_{\rel}(\sP) \Bigr]^\vir = \big[\oM_{\rel}(\sP/\BGm)\big]^\vir .
  \end{equation*}
  On the other hand Costello's Theorem \cite[Theorem~5.0.1]{C06}, applied to the left square above, tells us that
  \begin{equation*}
    \mu_\ast \Big[\oM_{\rel}(\sP / \BGm) \Big/ \oM_{\rel}(\sP) \Big]^\vir = [\fZ / \fM^\ast ]^\vir .
  \end{equation*}
\end{proof}

We are left to check the hypotheses of Costello's theorem and to prove Corollary~\ref{cor:absolutevirclass}.  In our situation, the hypotheses of Costello's theorem are (with bullets in order corresponding to the bullets of Costello's statement):
\begin{itemize}
\item $\oM_{\rm rel}(\sP / \BGm)$ and $\fZ$ are DM stacks:  immediate from the definitions;
\item $\oM_{\rel}(\sP)$ and $\fM^\ast$ are Artin stacks of the same pure dimension, and the bottom horizontal morphism is of DM type of pure degree 1:  in fact they are both Deligne--Mumford stacks by definition and the degree verification is done in Proposition~\ref{DMtype};
\item the top horizontal map is proper:  immediate from the properness of $\oM_{\rel}(\sP/\BGm)$;
\item the obstruction theories for the vertical maps agree:  we ouline our approach in the statement of Proposition~\ref{prop:gysin}, which will be demonstrated in Section~\ref{sec:OT}.
\end{itemize}

\begin{prop} \label{DMtype}
  The map $\oM_{\rel} (\sP) \rightarrow \fM^\ast$ induces an isomorphism on dense open substacks of source and target.
\end{prop}
\begin{proof}
  Let $U \subset \fM_{\rel}^\ast(\sP)$ be the locus of maps with unexpanded target.  To give such a map is precisely the same as to give a collection of disjoint, weighted sections on the source curve such that the sum of all the weights is zero.  Therefore the map in question induces an isomorphism from $U$ to its image in $\fM^\ast$.

  We must now argue that $U$ is dense in $\fM_{\rel}^\ast(\sP)$.  For this, let $\fM_{\rel}^{\orb}(\sP)$ be the stack of \emph{transverse} orbifold maps to root stacks of expansions of $\sP$, as considered in \cite{AF}.  By \cite[Lemma~3.2.6~(2)]{AF}, the stack $\fM_{\rel}^{\orb}(\sP)$ covers $\fM_{\rel}^\ast(\sP)$.  Therefore it suffices to see that the pre-image $U^{\orb}$ of $U$ in $\fM_{\rel}^{\orb}(\sP)$ is dense.  But the proof of \cite[Proposition~4.2.2]{ACW10} shows that $\fM_{\rel}^{\orb}(\sP)$ is smooth over the stack of orbifold targets.  Since the unexpanded orbifold targets are dense in the stack of all orbifold targets, this implies that $U^{\orb}$ is dense in $\fM_{\rel}^{\orb}(\sP)$.
\end{proof}

\begin{prop} \label{prop:gysin}
  Let $\sigma : Z \rightarrow J$ denote the inclusion.  The virtual class $[\oM_{\rel}(\sP / \BGm)]^\vir$ defined in \cite{GV05} is the Gysin pullback $\sigma^! [\oM_{\rel}(\sP)]$.
\end{prop}

The proof of this proposition will be given in Section~\ref{sec:OT}.  We note that $\sigma^! [\oM_{\rel}(\sP)]$ is the relative virtual class for $\oM_{\rel}(\sP/\BGm)$ associated to the relative obstruction theory over $\oM_{\rel}(\sP)$ that is pulled back from that of $Z$ in $J$.

\section{The obstruction theories}\label{sec:OT}

We will define and compare natural relative obstruction theories for the morphisms $Z \rightarrow J$ and $\oM_{\rel}(\sP / \BGm) \rightarrow \oM_{\rel}(\sP)$ of Diagram~(\ref{diag:costelloagain}).  Since the diagram is cartesian, any relative obstruction theory for the former morphism induces a relative obstruction theory for the latter.  This gives us two obstruction theories controlling the relative deformation theory of $\oM_{\rel}(\sP/\BGm)$ over $\oM_{\rel}(\sP)$.  The object of this section will be to show that these obstruction theories coincide.

We begin by explaining what we mean by an obstruction theory in Section~\ref{sec:obs}.  We have elected to use a definition that is close in spirit to that of~\cite{LT}, but incorporates some of the stack-theoretic techniques of~\cite{BF}.  This notion of an obstruction theory will be studied in detail in the forthcoming paper~\cite{obs}.  In the case of a perfect obstruction theory, the definition presented here is essentially equivalent to the one given in~\cite{BF}.

As in~\cite{BF}, the virtual fundamental class is obtained by intersecting the intrinsic normal cone with the zero section in a vector bundle stack associated to the obstruction theory.  However, unlike the obstruction theories considered in~\cite{BF}, the obstruction theories for relative stable maps introduced by J.\ Li do not obviously arise as $\Ext(\bE, -)$ for a complex $\fE$ of quasi-coherent sheaves on a scheme.  Although such a complex does exist a posteori, we do not know how to describe it directly.  This would make it difficult to verify the axioms of an obstruction theory using the Behrend--Fantechi formalism, and is the reason we have preferred the definition introduced below.

The difficulty of the obstruction theory is ultimately due to the predeformability condition, which is not open and therefore precludes standard deformation theoretic techniques.  In \cite{L02}, J.\ Li constructed his obstruction theory by measuring the failure of local deformations to glue using a modified \v{C}ech procedure.  We will reinterpret Li's obstruction groups as the groups of torsors under abelian group stacks on a suitable site, defined in Section~\ref{sec:site}.  This reinterpretation brought to light what appears to be a small omission in Li's original definition, so we have verified in detail in Appendix~\ref{app:obs} that our definition do give perfect obstruction theories.

In order to describe our obstruction theories, we will have to work systematically with abelian group stacks (or ``Picard stacks'' in the parlance of \cite[XVIII.1.4]{sga4-3}).  Verifying the axioms of an abelian group stack is tedious, though, and anyone who has done it once will shudder at the prospect of doing so repeatedly for the multiple abelian group stacks that appear in this paper.  Fortunately we have been able to rely on an elegant device due to Grothendieck~\cite{Gr} to avoid verifying the axioms directly:  we realize our abelian group stacks as fibers of additively cofibered categories, the fibers of which are always abelian group stacks (cf.\ \cite[Section~1.4]{Gr}).  

In Section~\ref{sec:unobs}, we describe the canonical obstruction theory associated to a smooth morphism.  This section is not used directly in the rest of the paper, but is meant to motivate the methods used to construct the obstruction theories considered in the rest of the section.  The principle is that whenever a deformation problem is locally trivial, it has a canonical obstruction theory coming from torsors under its tangent sheaf, which is a system of additively cofibered categories capturing essentially the same information as the cotangent bundle in the case of a smooth scheme.

In Section~\ref{sec:site}, we introduce the site in which deformations of relative stable maps become locally trivial in the present context.  Li demonstrates in \cite{L02} that the deformation theory of a relative stable map, relative to that of the source curve, is trivial \emph{\'etale locally on the source curve}.  Therefore one might expect that the deformation theory of a relative stable map is a local problem on the source curve, as it is for stable maps (see \cite{obs}).  However, a relative stable map is really a map into the \emph{fibers of a family} (to wit, the universal expansion $\tsP$ over the base $\tsT$ in the case considered here):  it contains the additional information of a map from the base of the family of source curves to the base of the family of targets.  To take this into account, our site combines the \'etale topologies of the base and total space of the family of curves.

In Section~\ref{sec:obs-rub}, we define the obstruction theories for $\oM_{\rel}(\sP/\BGm)$ and $\oM_{\rel}(\sP)$ relative to the stack of pre-stable curves.  Morally, these obstruction theories arise because these moduli problems can be extended to the site defined in Section~\ref{sec:site}, and they become formally smooth in that setting.  By the principle of local unobstructedness, we obtain canonical obstruction theories for the moduli problems that are necessarily compatible.  By comparing them, we obtain a relative obstruction theory for $\oM_{\rel}(\sP/\BGm)$ over $\oM_{\rel}(\sP)$.

Although the obstruction theories for $\oM_{\rel}(\sP/\BGm)$ and $\oM_{\rel}(\sP)$ over the moduli space of pre-stable curves are complicated and difficult to understand explicitly, their difference is much simpler.  We will find in Section~\ref{sec:leftarrow} that it is visibly the same as the obstruction theory pulled back from the normal bundle of $Z$ in $J$, supplying the final hypothesis of Costello's theorem and permitting us to conclude that 
\begin{equation*}
  \mu_\ast \Big[\oM_{\rel}(\sP/\BGm) \Big/ \fM_{\rel}^\ast(\sP)\Big]^{\vir} = [\fZ / \fM^\ast]^{\vir} ,
\end{equation*}
and completing the proof of Theorem~\ref{thm:compare}.

\begin{remark}
  The methods of this section are valid without the restriction to rational tails curves:  one may take $\oM_{\rel}(\sP/\BGm)$ and $\oM_{\rel}(\sP)$ to be the corresponding moduli spaces in which the source curves are allowed to be pre-stable.
\end{remark}

%
\subsection{Obstruction theories in general}
\label{sec:obs}

If $S$ is a scheme over an algebraic stack $X$ and $J$ is a quasi-coherent sheaf on $S$, let $\Def_X(S, J)$ be the category of square-zero extensions of $S$ by $J$ over $X$.  Objects of $\Def_X(S,J)$ are therefore diagrams
\begin{equation*} \xymatrix{
    S \ar[r] \ar[d] & S' \ar[dl] \\
    X 
  }
\end{equation*}
where $S'$ is a square-zero extension of $S$ with ideal $I_{S/S'} = J$.  This category has the following functoriality properties:
\begin{enumerate}
\item contravariance with \'etale morphisms in $S$:  if $f : S_1 \rightarrow S_2$ is \'etale and $J$ is a quasi-coherent sheaf on $S_2$ then there is a functor $\Def_X(S_2, J) \rightarrow \Def_X(S_1, f^\ast J)$;
\item covariance with $J$:  if $J_1 \rightarrow J_2$ is a morphism of quasi-coherent sheaves on $S$, there is a functor $\Def_X(S, J_1) \rightarrow \Def_X(S,J_2)$,
\item covariance with affine morphisms in $S$:  if $f : S_1 \rightarrow S_2$ is affine and $J$ is a quasi-coherent sheaf on $S_1$, there is a functor $\Def_X(S_1, J) \rightarrow \Def_X(S_2, f_\ast J)$.
\end{enumerate}
These functors are all compatible with composition in the usual sense (see \cite{obs} for details).
We note that the fibered category over the \'etale site of $S$ determined by $\Def_X$ is a stack and that for $S$ fixed, $\Def_X(S, J)$ is additively cofibered \cite[D\'efinition~1.2]{Gr} in the variable $J$.  This latter fact implies that $\Def_X(S,J)$ has the structure of a $\Gamma(S, \cO_S)$-$2$-module (the analogue for $\cO_S$-modules of what is called a ``Picard category'' in \cite[Expos\'e~XVIII~1.4]{sga4-3}) for all $S$ and $J$ (see \cite{gs} for details).  Only the abelian $2$-group structure will be relevant for us here, and for this one may refer to \cite[Section~1.4]{Gr}.

A representable morphism $X \rightarrow Y$ of algebraic stacks induces a faithful map $\Def_X(S,J) \rightarrow \Def_Y(S,J)$ for each $S$ and $J$.

\begin{defn}
  An \emph{obstruction theory} for $X$ over $Y$ is a collection of groupoids $\fE(S,J)$ for every scheme $S$ over $X$ and every quasi-coherent sheaf $J$ on $S$, such that:
\renewcommand{\labelenumi}{(\roman{enumi})}
\begin{enumerate}
\item $\fE(S,J)$ varies contravariantly with $S$ and covariantly with $J$;
\item $\fE(S,J)$ varies covariantly with affine morphisms in $S$;
\item $\fE$ is a stack on the big \'etale site of  $X$;
\item for $S$ fixed, $\fE(S,J)$ is additively cofibered and left exact in the $J$ variable;
\item there are given cartesian diagrams
  \begin{equation} \label{eqn:14} \xymatrix{
      \Def_X(S,J) \ar[r] \ar[d] & e(S,J) \ar[d] \\
      \Def_Y(S,J) \ar[r] & \fE(S,J)
    }
  \end{equation}
  where $e(S,J)$ is the zero object of the $2$-group $\fE(S,J)$;
\item the maps in Diagram~\eqref{eqn:14} are compatible with the \'etale contravariance in $S$, the affine covariance in $S$, and the covariance in $J$.
\end{enumerate}
\end{defn}
We limit ourselves to several remarks about this definition here and refer the reader to \cite{obs} for further details.  

\begin{rem}
  This definition is intermediate between those of \cite[Definition~4.4]{BF} and \cite[Definition~1.2]{LT}:  any obstruction theory in the sense of Behrend--Fantechi gives rise to one as above, which in turn gives rise to an obstruction theory in the sense of Li--Tian.  The composition of these proceses is the same as the one described in \cite[Section~3]{KKP} to produce a Li--Tian obstruction theory from a Behrend--Fantechi obstruction theory.
\end{rem}

\begin{rem}
  The diagram~\eqref{eqn:14} says roughly that associated to any square-zero extension $S'$ of $S$ over $Y$ (an object of $\Def_Y(S,J)$) there is an obstruction $\omega$ in $\fE(S,J)$ such that a lift of the diagram
\begin{equation*} \xymatrix{
    S \ar[r] \ar[d] & X \ar[d] \\
    S' \ar@{-->}[ur] \ar[r] & Y
  }
\end{equation*}
exists if and only if $\omega$ is isomorphic to the zero section $e(S,J)$.  Furthermore, the set of all such lifts is precisely the set of isomorphisms between $\omega$ and $e(S,J)$.
\end{rem}

\begin{rem}
Throughout, we will define various collections of categories and stacks depending on a scheme $S$ and a quasi-coherent sheaf $J$ on $S$, and satisfying various functoriality properties.  Although these objects will not generally be stacks (since they do not form fibered categories in the $S$ variable), it will be possible to obtain stacks by restriction to the small \'etale site of any given scheme $S$.  If $F$ is one of these objects, we will write $\uF(S, J)$ for the stack on the small \'etale site of $S$ whose value on $U$ is $F(U, J_U)$.

When $S$ and $J$ remain fixed and indicating the dependence on $S$ and $J$ seems more cumbersome than omitting it seems confusing, we will permit ourselves to write $\uF$ in place of $\uF(S,J)$.  
\end{rem}


%
\subsubsection{The virtual fundamental class}  Recall \cite[Definition~3.6 and Section~7]{BF} that the relative intrinsic normal sheaf $\fN_{X/Y}$ is the associated abelian cone stack $\ch(\bL_{X/Y}^\vee[1])$ of the dual of the relative cotangent complex of $X$ over $Y$.  Here, $\ch$ is Deligne's ``champ construction'' (cf.\ \cite[XVIII.1.4.11]{sga4-3} or \cite[Section~2]{BF}, where the notation $h^1 /h^0$ is used instead of $\ch$).   If $\fE$ is an obstruction theory in the sense above, then $S \mapsto \fE(S, \cO_S)$ is an abelian cone stack \cite{gs} that we will abusively denote by the same letter $\fE$.  We will say that $\fE$ is a perfect relative obstruction theory for $X$ over $Y$ if the abelian cone stack described above is a vector bundle stack \cite[Definition~1.9]{BF}.

If $\fE$ is a relative obstruction theory for $X$ over $Y$, there is a canonical embedding of abelian cone stacks $\fN_{X/Y} \rightarrow \fE$ \cite{obs}.  This induces an embedding of the relative intrinsic normal cone $\fC_{X/Y}$ in $\fE$.  If $\fE$ is a perfect relative obstruction theory then we can intersect $\fC_{X/Y}$ with the zero locus of $\fE$ to obtain the relative virtual class associated to this obstruction theory.


%
\subsubsection{Unobstructed morphisms}\label{sec:unobs}  This section will not be used in what follows.  It is provided to give intuition about the definition of an obstruction theory given above.

Consider a morphism $X \rightarrow Y$ and a commutative diagram of solid arrows
  \begin{equation} \label{eqn:18} \xymatrix{
      S \ar[r] \ar[d] & X \ar[d] \\
      S' \ar@{-->}[ur] \ar[r] & Y
    }
  \end{equation}
  in which $S'$ is a square-zero extension of $S$ with ideal $J$.  Let $T_{X/Y}(S,J)$ be the collection of completions of the diagram
  \begin{equation*} \xymatrix{
      S \ar[r] \ar[d] & X \ar[d] \\
      S[J] \ar[r]^0 \ar@{-->}[ur] & Y
    }
  \end{equation*}
  in which $S[J]$ is the trivial square-zero extension of $S$ by $J$ over $Y$ and the map $0 : S[J] \rightarrow Y$ is the zero tangent vector, i.e. the unique extension of $S \rightarrow Y$ through $S[J]$ that factors through the canonical retraction $S[J] \rightarrow S$.  If $X$ is of Deligne--Mumford type (resp.\ Artin type) over $Y$ then $T_{X/Y}(S,J)$ forms a $\Gamma(S, \cO_S)$-module (resp.\ a $\Gamma(S, \cO_S)$-$2$-module).  For $U$ \'etale over $S$, the assignment $U \mapsto T_{X/Y}(U,J_U)$ defines an abelian group stack $\uT_{X/Y}(S,J)$ on $S$.

  The letter $T$ is supposed to suggest the tangent bundle, which is justified by the equality $T_{X/Y}(S, J) = \Gamma(S, f^\ast T_{X/Y} \tensor_{\cO_S} J)$.  By Yoneda's lemma applied to the relative cotangent bundle, the system of modules (or $2$-modules) $T_{X/Y}(S,J)$ contains the same information as the relative tangent bundle itself.

  The dashed arrows completing Diagram~\eqref{eqn:18} form a pseudo-torsor under $T_{X/Y}(S,J)$ (this holds even under much weaker assumptions on $X$ and $Y$).  This is a consequence of the fact that algebraic stacks respect pushouts of infinitesimal extensions of schemes; the reader may find more details about this in \cite{obs}.  If $X$ is assumed to be smooth over $Y$ then, by the formal criterion of smoothness, this psuedo-torsor is a torsor under $\uT_{X/Y}(S,J)$.  The sections of this torsor are precisely the lifts of the diagram.
  
  If we define $\fE(S,J)$ to be the category of torsors under the sheaf of abelian groups (or stack of abelian $2$-groups) $\uT_{X/Y}(S,J)$ then $\fE(S,J)$ is a relative obstruction theory for $X$ over $Y$.

\begin{defn}
  The obstruction theory described above will be called the \emph{canonical relative obstruction theory} for $X$ over $Y$.  If $X$ is smooth over $Y$ and the relative obstruction theory is the canonical one, we say that $X$ is \emph{unobstructed} over $Y$.
\end{defn}
%
\subsubsection{Compatible obstruction theories} \label{sec:compatible}

Suppose $X \xrightarrow{u} Y \xrightarrow{v} Z$ is a sequence of morphisms of algebraic stacks, and $\fE_{X/Z}$ and $\fE_{Y/Z}$ are relative obstruction theories for $vu$ and $v$, respectively.  Suppose also that we have maps $\fE_{X/Z} \rightarrow u^\ast \fE_{Y/Z}$ and commutative diagrams
\begin{equation*} \xymatrix{
    \Def_X(S,J) \ar[dr] \ar[r] \ar[d] &  \Def_Z(S,J) \ar[r] \ar[dr] & \fE_{X/Z}(S,J) \ar[d]  \\
    \Def_Y(S,J) \ar[r] \ar[ur] & e(S,J) \ar[r] \ar[ur] & u^\ast \fE_{Y/Z}(S,J)
  }
\end{equation*}
in which the oblique parallelograms are the cartesian squares associated to the obstruction theories $\fE_{X/Z}$ and $\fE_{Y/Z}$.  Assume that the diagram above is compatible with the variation in $S$ and $J$.  Let $\fE_{X/Y}(S,J)$ be the kernel of $\fE_{X/Z}(S,J) \rightarrow u^\ast \fE_{Y/Z}(S,J)$~\cite{gs}.  Then $\fE_{X/Y}$ is naturally a relative obstruction theory for~$u$~\cite{obs}.  

\begin{defn}\label{def:compatible}
If, in addition, the map $\fE_{X/Z} \rightarrow u^\ast \fE_{Y/Z}$ is surjective as a map of \'etale stacks on $X$, we will say that these obstruction theories are \emph{compatible} and form an exact sequence
\begin{equation*}
  0 \rightarrow \fE_{X/Y} \rightarrow \fE_{X/Z} \rightarrow u^\ast \fE_{Y/Z} \rightarrow 0 .
\end{equation*}
\end{defn}


Suppose now that we have a sequence of compatible perfect obstruction theories as in Definition~\ref{def:compatible} and that $Y$ is locally unobstructed over $Z$, meaning that $Y$ is smooth over $Z$ and its relative obstruction theory is the canonical one.  Recall that by definition, $\fE_{X/Y}$ is the kernel of the map $\fE_{X/Z} \rightarrow u^\ast \fE_{Y/Z}$.  A section of $\fE_{X/Y}(S,J)$ is a pair $(\omega, \phi)$ where $\omega \in \fE_{X/Z}(S,J)$ and $\phi$ is an isomorphism between the image of $\omega$ in $\fE_{Y/Z}(S,J)$ and $e(S,J)$.  The kernel of $\fE_{X/Y} \rightarrow \fE_{X/Z}$ can therefore be identified with the group of automorphisms of the trivial object of $u^\ast \fE_{Y/Z}(S,J)$.  Recall, however, that $u^\ast \fE_{Y/Z}(S,J)$ is the $2$-stack of torsors under $u^\ast T_{Y/Z}(S,J)$, and the automorphism group of the trivial torsor is canonically $u^\ast T_{Y/Z}(S,J)$.  Moreover, the map $\fE_{X/Y} \rightarrow \fE_{X/Z}$ is surjective since any two sections of $u^\ast \fE_{Y/Z}$ are locally equivalent.  We have therefore proved the exactness of the bottom row of
\begin{equation*}
  \xymatrix{
    & u^\ast T_{Y/Z} \ar[r] \ar@{=}[d] & \fC_{X/Y} \ar[r] \ar[d] & \fC_{X/Z} \ar[d] \\
    0 \ar[r] & u^\ast T_{Y/Z} \ar[r] & \fE_{X/Y} \ar[r] & \fE_{X/Z} \ar[r] & 0 . 
  }
\end{equation*}
The compatibility of the obstruction implies the commutativity of the square on the right, in which $\fC_{X/Y}$ and $\fC_{X/Z}$ are the relative intrinsic normal cones \cite{BF}.  The inclusion of $\fC_{X/Y} \subset \fE_{X/Y}$ is equivariant with respect to the action of $T_{Y/Z}$, implying that $\fC_{X/Y}$ is the pre-image in $\fE_{X/Y}$ of $\fC_{X/Z} \subset \fE_{X/Z}$.  Furthermore, since $\fE_{X/Y}$ is a torsor over $\fE_{X/Z}$, the cycles $\fC_{X/Z} \subset \fE_{X/Z}$ and $\fC_{X/Y} \subset \fE_{X/Y}$ determine the same cycle class on $X$.  Therefore the virtual classes on $X$ associated to $\fE_{X/Y}$ and $\fE_{X/Z}$ must coincide.

In fact, \cite[Theorem~4]{Manolache} shows that the same conclusion holds if the hypothesis that $Y$ be unobstructed over $Z$ is replaced with the assumption that $Y$ be lci and of Deligne--Mumford type over $Z$, with the canonical relative obstruction theory.  Indeed, in that case the relative virtual class for $Y \rightarrow Z$ is the fundamental class of $Y$, whose virtual pullback to $X$ is therefore the same as the virtual pullback of the fundamental class of $Z$.

%
\subsection{A site adapted to deformations of curves}
\label{sec:site}

If $p:C \rightarrow S$ is a family Deligne--Mumford pre-stable curves, we can define a site $CS$ whose objects are commutative squares
\begin{equation*} \xymatrix{
    U \ar[r] \ar[d] & C \ar[d]^p \\
    V \ar[r] & S
  }
\end{equation*}
in which the horizontal arrows are \'etale.  For brevity, we sometimes refer to the above object of $CS$ as $U \rightarrow V$ or even just $UV$.  A collection of such diagrams is said to cover $C \rightarrow S$ if the maps $U \rightarrow C$ cover $C$ and the maps $V \rightarrow S$ cover $S$.

This site has a projection $\pi : CS \rightarrow \et(S)$.  The pullback functor $\pi^\ast : \et(S) \rightarrow CS$ sends a scheme $V$ that is \'etale over $S$ to $C \fp_{U} V \rightarrow V$.  There are embeddings $i : \et(C) \rightarrow CS$, with $i^\ast (UV) = U$, and $j : \et(S) \rightarrow CS$ with $j^\ast (UV) = V$.  Covers in $\et(C)$ (resp.\ in $\et(S)$) are all pulled back from covers in $CS$ so $R i_\ast = i_\ast$ (resp.\ $R j_\ast = j_\ast$).  From this it follows that $R \pi_\ast i_\ast = R p_\ast$ and $R \pi_\ast j_\ast = \id$.  We also have $j_\ast = \pi^\ast$, so $R \pi_\ast \pi^\ast = \id$ as well.

\begin{remark}
  The morphism of \'etale sites $\et(C) \rightarrow \et(S)$ induces a fibered site \cite[D\'efinition~VI.7.4.1]{sga4-2} over the category associated to the partially ordered set $\{ 0 \leq 1 \}$.  The site $CS$ defined above is the ``total site'' \cite[VI.7.4.3~3]{sga4-2} of this fibered site.
\end{remark}

By analogy, or using the remark above, one can define a site $XY$ as above for any morphism of sites $X \xrightarrow{f} Y$.  In addition to the situation considered above, we will also make use of the site $XY$ when $f$ is the morphism of sites associated to a finite morphism of schemes.

A sheaf on $XY$ can be viewed as a triple $(F, G, \varphi)$ where $F$ is a sheaf on $X$, $G$ is a sheaf on $Y$, and $\varphi : F \rightarrow f^\ast G$ is a morphism of sheaves on $X$.  We note that $j_! G = (0, G, 0)$ and $i_\ast F = (F, 0, 0)$, so for any sheaf $\sF$ on $XY$, there is an exact sequence
\begin{equation}  \label{eqn:22}
  0 \rightarrow j_! j^\ast \sF \rightarrow \sF \rightarrow i_\ast i^\ast \sF \rightarrow 0 .
\end{equation}

\begin{lem} \label{lem:exactness}
  Suppose that $p : X \rightarrow Y$ is a morphism of sites such that $p_\ast p^\ast = \id$ and $R p_\ast = p_\ast$.  Let $\pi : XY \rightarrow Y$ be the projection.  Then $\pi_\ast$ is exact on sheaves of abelian groups.
\end{lem}
\begin{proof}
  It is enough to show that $R^1 \pi_\ast \sF = 0$ for any sheaf of abelian groups $\sF$ on $XY$.  For this it is enough, by the exactness of~\eqref{eqn:22}, to see that $R^1 \pi_\ast (j_! G) = 0$ and $R^1 \pi_\ast (i_\ast F) = 0$ for all sheaves $F$ on $X$ and $G$ on $Y$.  For the second, note that $i_\ast$ is exact, and therefore $R^1 \pi_\ast (i_\ast F) = R^1 (\pi_\ast i_\ast) F = R^1 p_\ast F$, which is zero by hypothesis.

  Now consider $R^1 \pi_\ast ( j_! G)$.  We have an exact sequence
  \begin{equation*}
    0 \rightarrow j_! G \rightarrow \pi^\ast G \rightarrow i_\ast p^\ast G \rightarrow 0,
  \end{equation*}
  which is in fact the exact sequence~\eqref{eqn:22} applied to $\pi^\ast G = (p^\ast G, G, \id_{p^\ast G})$.  We have $R \pi_\ast \pi^\ast G = G$ and we have just seen that $R \pi_\ast i_\ast p^\ast G = R p_\ast p^\ast G$, which is $G$ by hypothesis.  Therefore $R \pi_\ast j_! G = 0$.
\end{proof}

\begin{lem} \label{lem:closed}
  Suppose $f : X \rightarrow X'$ is a closed embedding of $Y$-schemes.  Let $h : XY \rightarrow X'Y$ be the induced morphism.  Then $h$ is acyclic.
\end{lem}
\begin{proof}
  It is enough to prove this after localizing at a dense set of geometric points in $X'Y$.  We can therefore assume $Y$ is the spectrum of a strictly henselian local ring and $X'$ is either empty or the spectrum of a strictly henselian local ring, and with that assumption, we only need to show that the global sections functor on $XY$ is exact.  But in that case $X$ is either empty or the spectrum of a henselian local ring, and in either case $XY$ admits no covers other than by itself, which implies that the global sections functor must be exact.
\end{proof}

\subsection{Obstruction theories for rubber maps}
\label{sec:obs-rub}

In this section we will study an $S$-point of $\oM_{\rel}(\sP / \BGm)$, 
\begin{equation} \label{eqn:7} \xymatrix{
    C \ar[r] \ar[d]_p & \tsP \ar[d] \\
    S \ar[r] & \tsT .
  }
\end{equation}
We wish to describe an obstruction theory for $\oM_{\rel}(\sP/\BGm)$ relative to $\oM_{\rel}(\sP)$ at this $S$-point.  We will do this by comparing obstruction theories for each of these spaces relative to the moduli space of curves.

These obstruction theories are defined in Sections~\ref{sec:obthy1} and~\ref{sec:obthy2}, but we only prove that they are perfect relative obstruction theories in Sections~\ref{sec:obs-ax} and~\ref{sec:perfection}, since this fact is not directly relevant to our purpose in this section and relies on some technical calculations in Section~\ref{sec:calc}.  In Section~\ref{sec:leftarrow} we define the relative obstruction theory for $\oM_{\rel}(\sP/\BGm)$ over $\oM_{\rel}(\sP)$ using the method introduced in Section~\ref{sec:compatible}.

\subsubsection{Extension of the moduli problems} \label{sec:extension}
To construct the virtual class for rubber stable maps, we will use a technique inspired by J.\ Li's construction of the virtual class for stable maps to degenerations.  We were not able to understand Li's construction in its entirety so we have modified it slightly at several points.  

The guiding principle is that the obstruction theory should be obtained by gluing local obstruction theories that are canonical.  The non-uniqueness of the obstruction theory for a given moduli problem arises from the possibility of choosing different meanings for ``local''.  The site defined in Section~\ref{sec:site} provides a natural definition for ``local'' for the moduli problems at hand, and therefore the obstruction theories we define are in some sense canonical.  To define the obstruction theories, we effectively extend the moduli problems $\oM_{\rel}(\sP/\BGm)$ and $\oM_{\rel}(\sP)$ to this site, and define the natural obstruction theories in this setting, obtaining the global obstruction theories by gluing.  Although we will not use the extended moduli problems in an explicit way in what follows, they will nevertheless play an important background role; we therefore introduce the notation $\ufM_{\rel}(\sP/\BGm)$ and $\ufM_{\rel}(\sP)$ for them so we can make periodic comments about them.

Modulo the equivalence between torsors and \v{C}ech calculations, this is the same method used by Li in \cite{L02}, except Li's extension of the moduli problem is slightly ambiguous:  \cite[Lemma~1.12]{L02} is true with the definition of $\Hom(f^\ast \Omega_{W[n]}, I)^\dagger$ in \cite[p.~216]{L02} only for charts of first kind $p : U \rightarrow V$ where $p^{-1} \cO_V \rightarrow \cO_U$ is injective.  In general, the definition of $\Hom(f^\ast \Omega_{W[n]}, I)^\dagger$ must be modified slightly to yield Lemma~1.12 as stated (see Section~\ref{sec:calc} for some indications about this modification).  We note, however, that the first claim of \cite[Lemma~1.12]{L02} is unaffected by the discussion above, provided one chooses $U$ and $V$ to be small enough; we will use this claim essentially later in this section.

One way to remedy the ambiguity mentioned above is to modify the definition of charts of first kind and assume that $U$ always dominates $V$.  To make this change amounts to extending the moduli problems $\oM_{\rel}(\sP/\BGm)$ and $\oM_{\rel}(\sP)$ to $CS$ in a way that allows the expansion of $\sP$ to vary in a locally constant manner on the curve $C$ (instead of on the base $S$ as it normally does).  Although this extension does yield an obstruction theory, we do not expect this obstruction theory to be perfect.

We have selected the following natural extension of the moduli problem instead:  for $UV \in CS$, a $UV$-point of $\oM_{\rel}(\sP/\BGm)$ is a commutative diagram
\begin{equation*} \xymatrix{
    U \ar[r] \ar[d] & \tsP \ar[d] \\
    V \ar[r] & \tsT
  }
\end{equation*}
that is non-degenerate and predeformable.  This is equivalent in principle to redefining $\Hom(f^\ast \Omega_{W[n]}, I)^\dagger$ to be the collection of dashed arrows completing the diagram
\begin{equation*} \xymatrix{
  U \ar[r] \ar[d] \ar@/^15pt/[rr] & U[p^\ast I_V] \ar@{-->}[r] \ar[d] & \tsP_{W[n]} \ar[d] \\
  V \ar[r] \ar@/_15pt/[rr] & V[I_V] \ar@{-->}[r] & W[n] .
} \end{equation*}
As was remarked already, this does not change $\Hom(f^\ast \Omega_{W[n]}, I)^\dagger$ when $U \rightarrow V$ is surjective.

%
\subsubsection{The obstruction theory for $\oM_{\rel}(\sP/\BGm)$} \label{sec:obthy1} 
We describe the relative obstruction theory~$\fE$ for~$\oM_{\rel}(\sP/\BGm)$ over~$\fM$.  Let $J$ be a quasi-coherent sheaf on $S$.  For $UV \in CS$ and $J$ a quasi-coherent sheaf on $V$, define $T(UV, J)$ to be the category of predeformable completions of the diagram
\begin{equation*} \xymatrix{
    U \ar[r] \ar[d] \ar@/^15pt/[rr] & U[p^\ast J] \ar@{-->}[r] \ar[d] & \tsP \ar[d] \\
    V \ar[r] \ar@/_10pt/[rr] & V[J] \ar@{-->}[r] & \tsT
  }
\end{equation*}
Then $T(UV,J)$ is additively cofibered with respect to $J$, so each $T(CS,J)$ is an abelian $2$-group.  Allowing $UV$ to vary in $CS$ we obtain an abelian group stack $\uT(CS, J)$ for each quasi-coherent sheaf $J$ on $S$.  When $CS$ and $J$ are fixed we will permit ourselves to drop them from the notation and write $\uT$ for the corresponding sheaf on $CS$.

Consider an extension problem
\begin{equation}  \label{eqn:8} \xymatrix{
    C \ar[r] \ar[d] \ar@/^10pt/[rr] & C' \ar@{-->}[r] \ar[d] & \tsP \ar[d] \\
    S \ar[r] \ar@/_10pt/[rr] & S' \ar@{-->}[r] & \tsT \\
  }
\end{equation}
in which $C'$ is flat over $S'$ and $I_{S/S'} = J$ and we search for \emph{predeformable} dashed arrows rendering the whole diagram commutative.  Solutions to this problem form a pseudo-torsor under $T(CS,J)$.  This follows from the fact that Artin stacks respect pushouts of infinitesimal extensions of schemes and the pushout of two predeformable morphisms is still predeformable.  By \cite[Lemma~1.12]{L02}, a solution exists locally in $CS$, so the stack on $CS$ of solutions to~\eqref{eqn:8} is a torsor under the abelian group stack $\uT(CS,J)$.  Denote by $\fE(S,J)$ the category of torsors under $\uT(CS,J)$ on $CS$.  Note that $\fE(S,J)$ is a $2$-category by definition, but it is equivalent to a $1$-category because the identity section of $T(CS,J)$ has no \emph{global} automorphisms on $CS$, by stability.

\begin{rem}
  The stack $\uT(CS,J)$ may be viewed as the relative tangent bundle for a morphism of ``refined'' moduli spaces $\ufM_{\rel}(\sP/\BGm) \rightarrow \ufM$ in which the hypotheses about connectedness of the fibers and properness in the moduli problems are relaxed.
\end{rem}

An object of $\Def_{\fM}(S, J)$ corresponds to a diagram
\begin{equation*} \xymatrix{
    C \ar[r] \ar[d] & C' \ar[d] \\
    S \ar[r] & S' .
  }
\end{equation*}
in which $I_{S/S'} = J$.  Any such extension gives rise to a lifting problem~\eqref{eqn:8}, hence to a torsor on $CS$ under the abelian group stack $\uT(CS,J)$.  Trivializations of this torsor correspond precisely to solutions to the lifting problem~\eqref{eqn:8}.  We therefore obtain a cartesian diagram
\begin{equation*} \xymatrix{
    \Def_{\oM_{\rel}(\sP/\BGm)}(S,J) \ar[r] \ar[d] & e(S,J) \ar[d] \\
    \Def_{\fM}(S,J) \ar[r] & \fE(S,J) .
  }
\end{equation*}
This verifies one axiom of an obstruction theory for the map $\oM_{\rel}(\sP/\BGm) \rightarrow \fM$; the remaining axioms are checked in Section~\ref{sec:obs-ax}.  Its perfection is checked in Section~\ref{sec:perfection}.




%
\subsubsection{The obstruction theory for $\oM_{\rel}(\sP)$} \label{sec:obthy2}
The obstruction theory $\fE'$ for $\oM_{\rel}(\sP)$ over $\fM$ is described similarly, with $\sT^2$ replacing $\tsT$.  Solutions to the lifting problem
\begin{equation}  \label{eqn:9} \xymatrix{
    C \ar[r] \ar[d] \ar@/^15pt/[rr] & C' \ar@{-->}[r] \ar[d] & \tsP \ar[d] \\
    S \ar[r] \ar@/_15pt/[rr] & S' \ar@{-->}[r] & \sT^2 .
  }
\end{equation}
form a pseudo-torsor under the abelian group stack of completions of the diagram
\begin{equation}  \label{eqn:10} \xymatrix{
    C \ar[r] \ar[d] \ar@/^15pt/[rr] & C[p^\ast J] \ar@{-->}[r] \ar[d] & \tsP \ar[d] \\
    S \ar[r] \ar@/_15pt/[rr] & S[J] \ar@{-->}[r] & \sT^2 .
  }
\end{equation}
We denote this abelian group category by $T'(CS,J)$ which extends, as before, to a sheaf $\uT'(CS,J)$ on $CS$.  Take $\fE'(S,J)$ to be the category of torsors on $CS$ under $\uT'(CS,J)$.  Then, as above, $\fE'$ is a relative obstruction theory for the map $\oM_{\rel}(\sP)\rightarrow \fM$.

%
\subsubsection{The relative obstruction theory for $\oM_{\rel}(\sP/\BGm)$ over $\oM_{\rel}(\sP)$} \label{sec:leftarrow}

Working with a fixed $S$-point~\eqref{eqn:2} of $\oM_{\rel}(\sP/\BGm)$ and a fixed quasi-coherent sheaf $J$ on $S$, we shall write $\uT$ for $\uT(CS,J)$ and $\uT'$ for $\uT'(CS,J)$.

Recall the isomorphism $\tsT\simeq\sT^2\times B\Gm$ from Proposition~\ref{prop:1}.  Composition with the projection $\tsT \rightarrow \sT^2$ induces a map $\uT \rightarrow \uT'$.  Let $\uT''$ be the quotient $\uT / \uT'$ in the sense of abelian group stacks, i.e. the associated stack of $(S,J) \mapsto T(S,J) / T'(S,J)$.  Set $\fE''(S,J) = T''(CS,J)$.  This is just $\pi_\ast \uT''$ where  $\pi : CS \rightarrow \et(S)$ is the projection from Section~\ref{sec:site}.  We wish to understand $\fE''$ more explicitly and see that it is a relative obstruction theory for $\oM_{\rel}(\sP / \BGm) \rightarrow \oM_{\rel}(\sP)$.

\begin{prop} \label{prop:ext}
  The sequence
  \begin{equation} \label{eqn:16}
    0 \rightarrow \fE'' \rightarrow \fE \rightarrow \fE' \rightarrow 0
  \end{equation}
  is exact and forms a compatible sequence of obstruction theories for the sequence of maps
  \begin{equation*}
    \oM_{\rel}(\sP/\BGm) \rightarrow \fM_{\rel}(\sP) \rightarrow \fM.
  \end{equation*}
  Furthermore, there is a natural equivalence $\fE''(S, J) \simeq \Ext(\bE, J)$ where $\bE$ is the Hodge bundle.
\end{prop}     
\begin{proof}
  The left exactness of the sequence~\eqref{eqn:16} is immediate from the definition of $\fE''$.  This implies that $\fE''$ is a relative obstruction theory for $\oM_{\rel}(\sP/\BGm)$ over $\oM_{\rel}(\sP)$, as in Section~\ref{sec:compatible}.  Compatibility of the obstruction theories will thus follow from surjectivity of the map $\fE \rightarrow \fE'$.  Note, however, that the lifts of any given section form a $\uT''$-torsor on $CS$.  It is therefore enough to show that all such torsors are trivial relative to $S$.  Along the way, we will see that $\fE''(S,J) = \Hom(\bE, J)$ where $\bE$ is the Hodge bundle of $C$.


  First we calculate $\pi_\ast \uT''$ explicitly.  Let $L(CS,J)$ be the space of extensions
  \begin{equation*} \xymatrix{
      S \ar[d] \ar[r] & \BGm \\
      S[J] \ar@{-->}[ur],
    }
  \end{equation*}
  and let $L'(CS,J)$ be the space of extensions
  \begin{equation} \label{eqn:13} \xymatrix{
      C \ar[r] \ar[d] & \BGm \\
      C[p^\ast J] \ar@{-->}[ur].
    }
  \end{equation}
  These extend to abelian group stacks $\uL$ and $\uL'$ on $CS$ and we have a commutative diagram
  \begin{equation*} \xymatrix{
      \uT \ar[r] \ar[d] & \uT' \ar[d] \\
      \uL \ar[r] & \uL'
    }
  \end{equation*}
  coming from the isomorphism $\tsT \simeq \sT^2 \times \BGm$.  In fact, this diagram is cartesian, so if we define $\fF(S,J)$ to be the $2$-category of $\uL(CS,J)$-torsors on $CS$ and $\fF'(S,J)$ to be the $2$-category of $\uL'(CS,J)$-torsors on $CS$, the diagram
  \begin{equation} \label{eqn:17} \xymatrix{
      \fE \ar[r] \ar[d] & \fE' \ar[d] \\
      \fF \ar[r] & \fF'
    }
  \end{equation}
  is also cartesian.  

  Let $\uL''$ be the quotient abelian group stack $\uL' / \uL$.  This is an abelian $3$-group on $CS$, but we will only be interested in $\fF'' = \pi_\ast L''$, which is an abelian $2$-group because an automorphism of a line bundle on $S$ is determined \emph{globally} by its pullback to $C$.  We can identify $\fF''$ with the kernel of $\fF \rightarrow \fF'$.  Since $\fE''$ was defined analogously and Diagram~\eqref{eqn:17} is cartesian, the induced map $\fE'' \rightarrow \fF''$ must be an equivalence.

We follow the notation of \cite[XVIII.1.4]{sga4-3} and use $\ch$ to denote the assignment which takes a $2$-term complex concentrated in degrees $[-1,0]$ to its associated Picard stack.  It is easy to see that $\uL(CS, J)$ is represented on $CS$ by $\ch(j_\ast J[1])$ and $\uL'(CS, J)$ is represented by $\ch(i_\ast p^\ast J[1])$ (see Section~\ref{sec:site} for the definitions of $i$ and $j$).  It is also immediate to verify that the map $\uL \rightarrow \uL'$ is represented by the natural map $j_\ast J[1] \rightarrow i_\ast p^\ast J[1]$ induced from the map $J \rightarrow p_\ast p^\ast J$.  If we take $K$ to be the cone of the map $J[1] \rightarrow R p_\ast p^\ast J[1]$, we obtain an exact triangle
\begin{equation*}
  \cO_S[1] \rightarrow R p_\ast \cO_C[1] \rightarrow K \rightarrow \cO_S[2]
\end{equation*}
and we have $\ch(K) = \fF''$.  Since $J \rightarrow p_\ast p^\ast J$ is an isomorphism by Zariski's main theorem, we can also see that $K$ is isomorphic to $R^1 p_\ast J[0] = \Ext^1(\bE, J)$, where $\bE$ is the Hodge bundle.  This induces a functorial isomorphism $\fE''(S, \cO_S) \simeq \Gamma(S, \bE^\vee)$.  

It follows, furthermore, from the above that torsors on $CS$ under $\uT''(CS, J)$ are parameterized up to isomorphism by $R^2 p_\ast J$, which is zero since $C$ is $1$-dimensional over $S$, so all $\uT''(CS, J)$-torsors on $CS$ are trivial.  This gives the surjectivity of $\fE \rightarrow \fE'$, since the fiber above any section of $\fE'$ is a $\uT''$-torsor on $CS$.
\end{proof}

In particular, $\ufE''(S, \cO_S)$, which was a priori an abelian group stack, is actually just a sheaf of abelian groups, in the usual sense.  This proves that $\fE''$ at least has the same underlying vector bundle as the obstruction theory pulled back from that of $Z$ in $J$.

The one thing remaining to be checked to deduce that the obstruction theory of $\oM_{\rel}(\sP/\BGm)$ over $\oM_{\rel}(\sP)$ is pulled back from that of $Z$ over $J$ is that the \emph{obstructions} coming from our two obstruction theories are the same.  Consider a lifting problem
\begin{equation*} \xymatrix{
    S \ar[r] \ar[d] & \oM_{\rel}(\sP/\BGm) \ar[r] \ar[d] & \tsT \ar[d] \\
    S' \ar[r] \ar@{-->}[ur] & \oM_{\rel}(\sP) \ar[r] & \sT^2 
  }
\end{equation*}
in which $S'$ is a square-zero extension of $S$ with $I_{S/S'} = \cO_S$.  Let $C$ and $C'$ be the curves over $S$ and $S'$ associated to the maps $S \rightarrow \oM_{\rel}(\sP/\BGm)$ and $S' \rightarrow \oM_{\rel}(\sP)$.  The lifting problem above translates immediately into the extension problem
\begin{equation*} \xymatrix{
    C \ar[r] \ar[d] & C' \ar[r] \ar[d] & \BGm \\
    S \ar[r] \ar@/_30pt/[urr] & S' \ar@{-->}[ur]
  }
\end{equation*}
and the obstruction is precisely the class of the line bundle associated to the map $C' \rightarrow \BGm$ in the dual of the Hodge bundle.  This is the same as the obstruction defined earlier coming from the inclusion of $Z$ in $J$.  This completes the proof.


\begin{prop}
The map $\oM_{\rel}(\sP)\rightarrow \fM$ is lci and $\fE'$ coincides with the canonical relative obstruction theory.
\end{prop}
\begin{proof}
  We have seen that $\oM_{\rel}(\sP)$ has a smooth, dense open substack $U$ that is isomorphic to the stack parameterizing smooth curves with $T$ disjoint sections with weights summing to zero.  By Proposition~\ref{prop:perfect}, $\fE'$ is a perfect relative obstruction theory for $\oM_{\rel}(\sP)$ over $\fM$.  It is shown in \cite{obs} that any perfect obstruction theory that is locally of finite presentation (in the sense of Proposition~\ref{prop:lfp-E}) arises from a perfect obstruction theory in the sense of Behrend--Fantechi.  Therefore, by \cite[Lemma~B.2]{ACW10}, it now suffices to check that this obstruction theory is the canonical one on the open substack $U$.

  Let $p : C \rightarrow S$ be a smooth curve over $S$ and suppose given a map $f : C \rightarrow \sP$.  To give a map $C \rightarrow \sP$ is the same as to give a pair of disjoint divisors in $C$.  This holds without any assumption of properness on $C$, and therefore also holds for the open subsets of $C$.  

  Let $D$ be the union of the pre-images of $0$ and $\infty$.  For any quasi-coherent sheaf $J$ on $S$, the extensions of $f$ to a map $C[p^\ast J] \rightarrow \sP$ are the same as the extensions of $D$ to a divisor on $C[p^\ast J]$, which are parameterized by $\Gamma(D, N_{D/C} \tensor J)$.  In other words, the sheaf $T'$ of extensions on $CS$ is pushed forward by the closed embedding $h : D \rightarrow CS$.  Let $\psi : D \rightarrow S$ be the projection.  Then $h_\ast$ is exact since $h$ is a closed embedding, so $R^1 \pi_\ast \uT' = R^1 \psi_\ast h^\ast \uT'$.  On the other hand $R^1 \psi_\ast = 0$ since $D$ is finite over $S$.  Therefore the obstructions for $\oM_{\rel}(\sP)$ over $\fM$ all vanish locally in $S$, which is exactly what is needed.
\end{proof}

\begin{cor}\label{cor:absolutevirclass}
  The virtual class for $\oM_{\rel}(\sP / \BGm)$ relative to $\oM_{\rel}(\sP)$ is equal to the absolute virtual class.
\end{cor}
\begin{proof}
  By the compatibility of the obstruction theories, the relative class for $\oM_{\rel}(\sP / \BGm)$ over $\oM_{\rel}(\sP)$ coincides with the virtual pullback of the virtual class of $\oM_{\rel}(\sP)$ relative to $\fM$.  But since $\fE'$ is the canonical obstruction for the lci morphism $\oM_{\rel}(\sP) \rightarrow \fM$, the relative virtual class must be the fundamental class.
\end{proof}

\section{Localization}
In this section, we summarize the localization techniques we use and prove Lemma~\ref{lem:psiclass}, which evaluates a specific product of $\psi$ classes.  We refer the reader to \cite[Part~2]{GJV06} and \cite[Sections~0 and~1]{FP05} for a thorough treatment of relative virtual localization in the case of $\PP^1$.  Localization in the context of stable maps was introduced by Kontsevich in \cite{K95}.  Virtual localization was derived in \cite{GP99} and later for relative stable maps in \cite{GV05}.

\subsection{Virtual localization}
\label{sec:vl}

Let $X$ be a  proper Deligne-Mumford stack with a $\C^\ast$-equivariant perfect obstruction theory. Let $\iota_i:F_i\to X$ be the irreducible components of the fixed locus by the torus action. Then, for any ${[\sK]}\in \Chow_{\C^\ast}(X)$:
\begin{equation}
\label{gpvl}
{[\sK]}=\sum_i \iota_{i\ast}\frac{{[\sK]}_{|F_i}}{e(N^{vir}_{F_i})} \in \Chow_{\C^\ast}(X)\otimes_{\C[t]}\C(t),
\end{equation}
where $e(N^{vir}_{F_i})$ is the virtual equivariant Euler class of the virtual normal bundle to the fixed locus, as described in \cite{GP99}.
We apply (\ref{gpvl}) to spaces of relative stable maps to $\PP^1$.

\subsection{Localization and relative maps to $\PP^1$}Let $\C^\ast$ act on $\PP^1$ with fixed points $0$ and $\infty$ and weights $1$  on $T_0\PP^1$ and and $-1$ on $T_\infty\PP^1$. Deonte by $t$  the first Chern class of the standard representation of $\C^\ast$, so that $\Chow^\ast_{\C^\ast}(\text{pt})=\Q[t]$. The action on $\PP^1$ induces an action of $\C^\ast$ on $\overline{\sM}_{g,n}(\PP^1, d\infty)$ (hereafter denoted $\rma{g,n}{d}$) and the virtual fundamental class $\virclass{\rma{g,n}{d}}$ has a canonical equivariant lift.  There is a branch morphism $\text{br}: \rma{g,n}{d} \to \text{Sym}^{r_\alpha^g}\PP^1$ which is  equivariant with respect to the induced action on the target.  A $\C^\ast$ fixed relative stable map $[C\to T \to \PP^1]$ (we write $T\to\PP^1$ for the contraction of the expanded components of the accordion $T$) over $\text{Spec } \C$ is characterized by the following properties:
\begin{figure}[bt]
\begin{center}
		\includegraphics[width=0.8\textwidth]{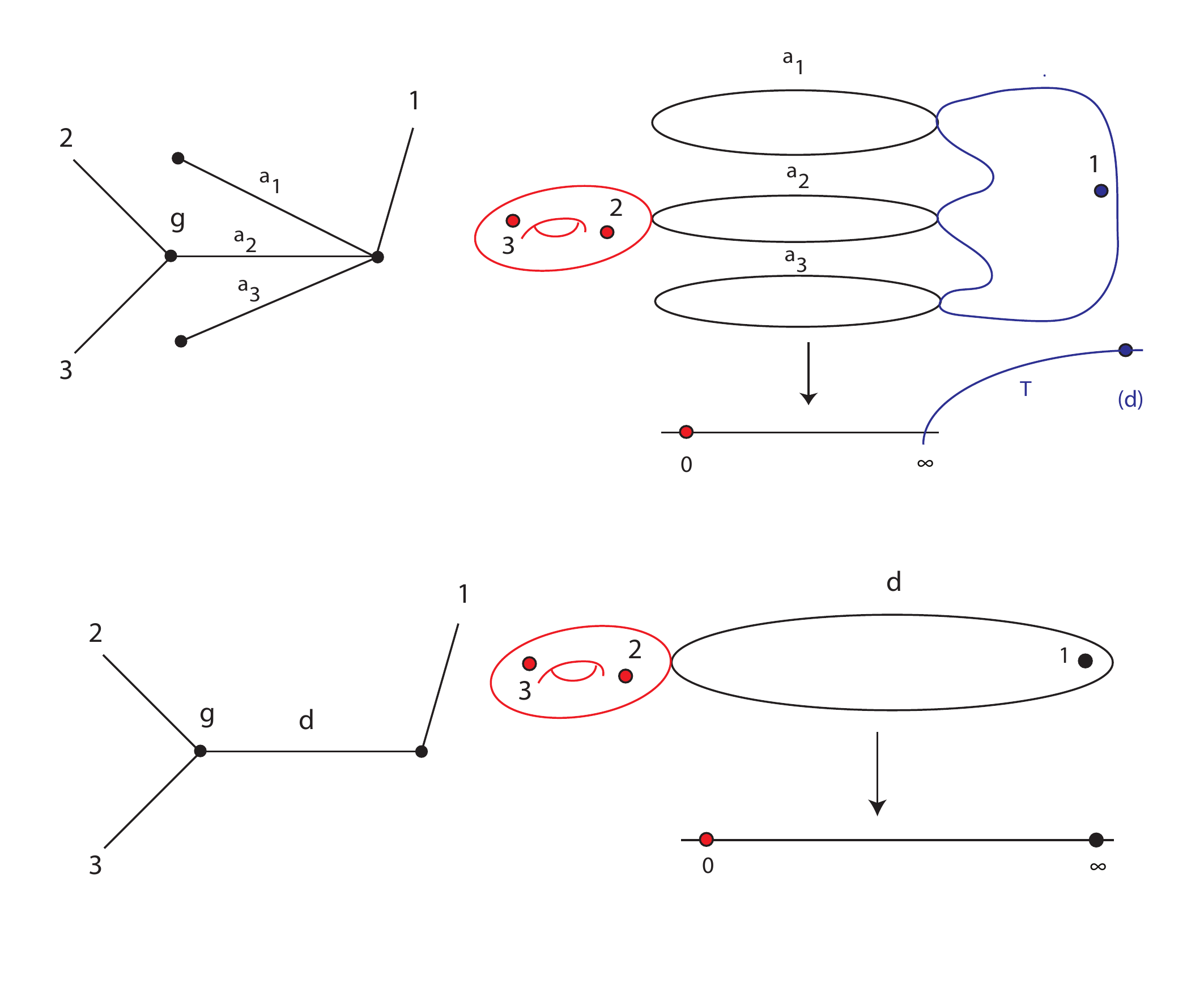}
\caption{Two torus fixed maps and their corresponding graphs}
\label{example1}
\end{center}
\end{figure}

\begin{enumerate}
\item any irreducible component of $C$ mapping surjectively to the base $\PP^1$ is a rational curve, fully ramified over the fixed points $0$ and $\infty$;
\item all other components must contract onto $0$ or map into the expanded components of $T$ which  contract onto $\infty$;
\item all markings lie over $0$ or $\infty$.
\end{enumerate} 
These features are illustrated in Figure \ref{example1}.
Since $P_g(\alpha;d)$ is defined by restricting to the rational tails locus, we only consider maps where the curve $C$ is rational tails.

The fixed loci of the $\C^\ast$-action on $\rma{g,n}{d}$ are indexed by localization graphs (see \cite[Section~1.3.2]{FP05}).  The torus fixed locus corresponding to a localization graph $\Gamma$ is isomorphic to $\cms_\Gamma / \A_\Gamma$, where $\A_\Gamma$ and its action on $\cms_\Gamma$ are described in \cite[Section~1.3.4]{FP05}.  There each $\Gamma$ is also assigned a multiplicity $m(\Gamma)$.  The relative virtual localization formula then says
\begin{equation}
\label{virloc}
\virclass{\rma{g,n}{d}}=\sum_{\Gamma}\frac{m(\Gamma)}{|\A_\Gamma|}{\iota_\Gamma}_\ast\left(\frac{\virclass{\cms_\Gamma}}{e_{\C^\ast}(N_\Gamma^{\text{vir}})}\right)
\end{equation}
in the ring $\A^{\C^\ast}_\ast(\rma{g,n}{d})\otimes\Q[t,\frac{1}{t}]$.  Here $N_\Gamma^{\text{vir}}$ is the virtual normal bundle to the fixed locus $\Gamma$ in $\rma{g,n}{d}$ and we are taking its equivariant Euler class.  There is an explicit formula for $\frac{1}{e_{\C^\ast}(N_\Gamma^{\text{vir}})}$ in \cite{GJV06} at the end of Section 3.4. 

Since all the genus of a rational tails curve is concentrated in a single irreducible component, all our contributing graphs are trees.  In fact, they are bipartite ``fans" since being relative to $\infty$ with full ramification $(d)$ allows for only a single connected component contracting to $\infty$.  We introduce notation for localization graphs of two types, depending on whether the genus $g$ component of the source curve lies above $0$ or $\infty$. 

Let $[C\to T\to \PP^1]$ be a $\C^\ast$ fixed map.  When the genus $g$ component lies above $0$, the contributing localization graphs $\Gamma_L(\nu,j)$  (see Figure \ref{example2}) are parameterized by the data of a partition $\nu=\nu_1+\ldots+\nu_l\vdash d$ for the degree of the edges and an integer $j\in\{1,\ldots,l\}$ indicating to which edge the genus $g$ component attaches. We include the data of the marked points in our notation by a subscript corresponding to each marked point placed on the corresponding part of $\nu$.   

\begin{figure}[h]
\begin{center} 
		\includegraphics[width=0.8\textwidth]{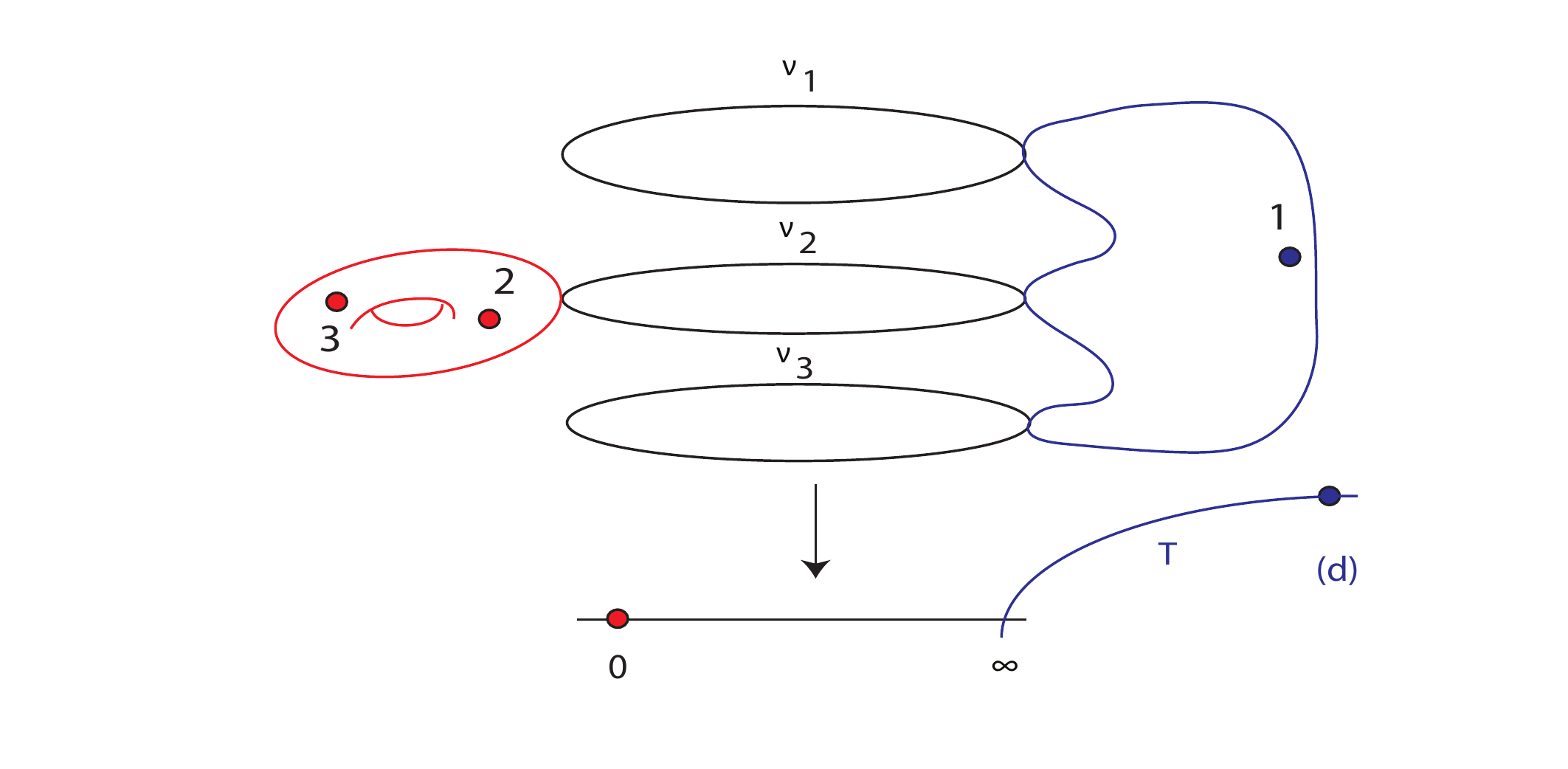}
\caption{Fixed map corresponding to the localization graph $\Gamma_L(\nu_1+({\nu_2})_{2,3}+\nu_3,2)$.}
\label{example2}
\end{center}
\end{figure}

These fans correspond to the fixed locus information:
\begin{align*}
\cms_\Gamma=\cms_{g,\text{val}(v_j)}\times\abrub{0}{\mu}{d} \times \prod_{i\neq j}\cms_{0,\text{val}(v_i)} & &m(\Gamma)=\prod_i\mu_i,
\end{align*}
where we share the abuse of notation of \cite{FP05} and intend $\cms_{0,1}=\cms_{0,2}:= pt$.

When the genus $g$ component lies above $\infty$, the contributing localization graphs $\Gamma_R(\nu)$ are parameterized only by the data of the partition $\nu\vdash d$ and the location of the marked points, as illustrated in Figure \ref{example3}.

\noindent The graph $\Gamma_R(\nu)$ corresponds to:
\begin{align*}
\cms_\Gamma=\abrub{g}{\mu}{d}\times\prod_{i}\cms_{0,\text{val}(v_i)}& &m(\Gamma)=\prod_i\mu_i
\end{align*}

\begin{figure}[h]
\begin{center}
		\includegraphics[width=0.8\textwidth]{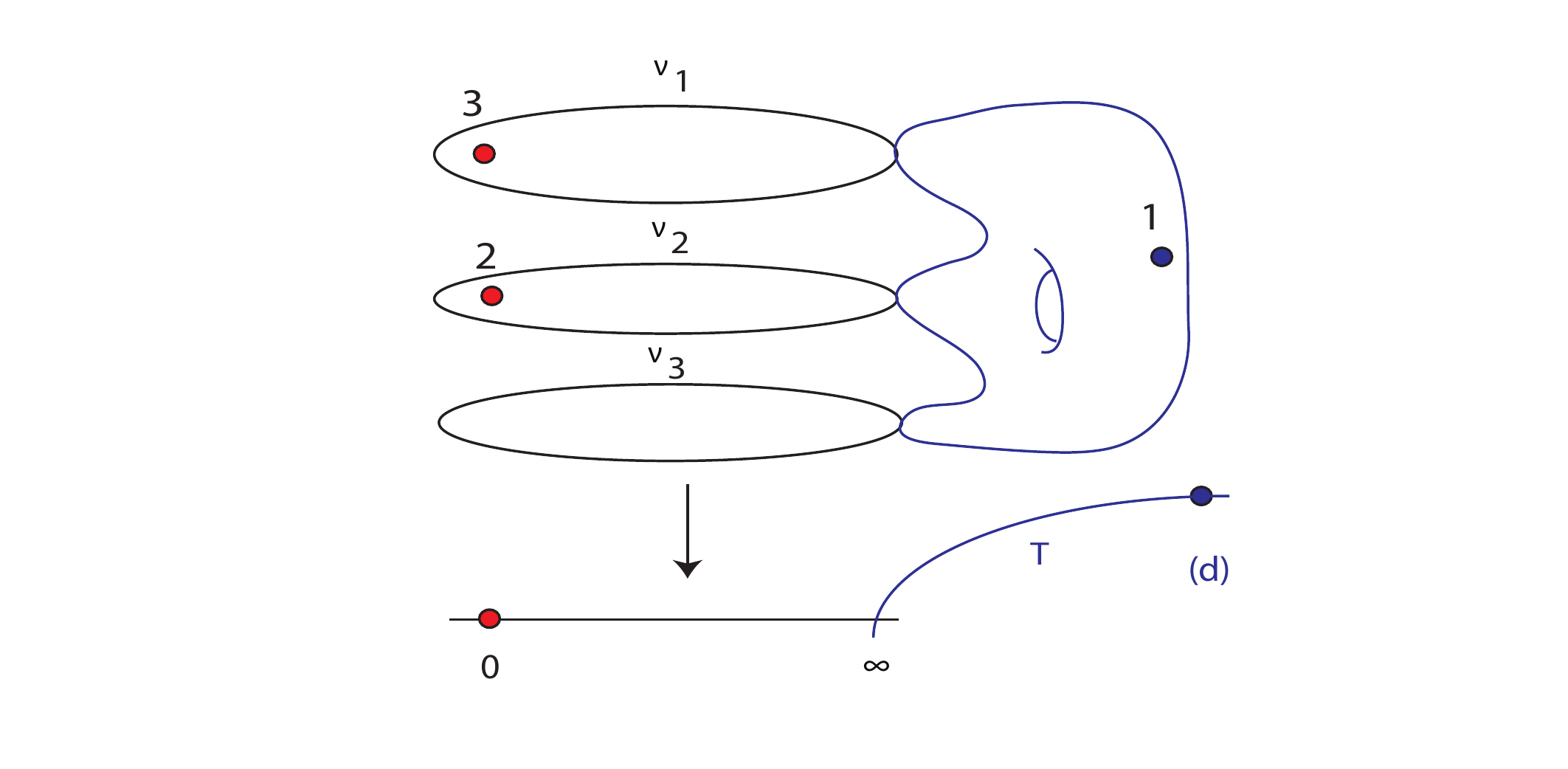}
\caption{Fixed map corresponding to the localization graph $\Gamma_R(({\nu_1})_3+({\nu_2})_{2}+\nu_3)$.}
\label{example3}
\end{center}
\end{figure}

\subsection{Cotangent line bundle classes from the target}\label{psiclasses}  The contribution of the equivariant Euler class of the virtual normal bundle in formula (\ref{virloc}) involves ``$\psi$" classes of two types.
  Smoothing nodes over $0$ gives ordinary $\psi_i$ classes on moduli spaces of curves. Smoothing nodes over $\infty$  produces a ``$\psi$ class below", which we denote by $\psi$ without a subscript, on a space of rubber relative maps $\abrub{0}{\alpha}{\beta}$.  This corresponds to the pullback of an appropriate $\psi$ class via the branch morphism $\abrub{0}{\alpha}{\beta} \to \cms_{0,r+2}/S_r$ (see \cite[Section 2.5]{GV05}).

\begin{lem}[\cite{GJV06}]
  Let $r = r^0_{\alpha,\beta} = l(\alpha)+l(\beta)-2$ be the number of simple ramification points for a map in  $\abrub{0}{\alpha}{\beta}$ and let $H^0_{\alpha,\beta}$ be the genus zero double Hurwitz number as defined in \cite[Section 1.4.1]{gjv:dhn}.  Then
\label{lem:psiclass}
\[
\int_{\virclass{\abrub{0}{\alpha}{\beta}}}\psi^{r-1}=\frac{1}{r!}H^0_{\alpha,\beta} .
\]
\end{lem}
\begin{proof}
We refer the reader to pages 21-22 of \cite{GJV06}.
\end{proof}

%
%
\section{Total Length 2}
In this section we compute $P_{g,2}(d;d)$, providing an independent proof of Theorem \ref{theorem:T=2}.  The class $P_{g,2}(d;d)$ has codimension $2g-1$ and lives in the socle of $R^\ast(\sM^{rt}_{g,2})$.  The morphism $\pi:\sM^{rt}_{g,2}\to\sM^{rt}_{g,1}$ forgetting the second marked point induces an isomorphism of socles. Computing $\pi_\ast(P_g(d,d))$ is then equivalent to computing the original class, and we abuse notation by omitting $\pi_\ast$ when working over $\sM^{rt}_{g,1}$. This shortcut has the threefold advantage of simplifying the combinatorics of our auxiliary integrals, of making explicit the comparison with Theorem 3.5 of \cite{GJV06}, and of allowing us to use a generating function for Hodge integrals on $\cms_{g,1}$ from \cite{FP00}.  
\begin{table}[t] 
\begin{tabular}[c]{| T || T | S |}
\hline
 & 1. \T\B & 2.\\ \hline
\hline
\multicolumn{3}{|c|}{Fixed Locus Information\T\B}\\
\hline
\hline
graph&$\Gamma_R(d)$&$\Gamma_{L\B}(\nu;j)\T$\\
\hline
$\frac{m(\Gamma)}{|\A_\Gamma|}$& $\frac{d}{d}$& $\displaystyle\frac{(\prod_{i=1}^k\nu_i^{m_i})m_j\T}{(\prod_{i=1}^k\nu_i^{m_i})(\prod_{i=1}^k m_i!)\B}$\\
\hline
fixed locus & $\abrub{g}{d}{d}$ \T\B& $\cms_{g,1}\times\abrub{0}{\nu}{d}$ \\
\hline
\hline
\multicolumn{3}{|c|}{Normal Bundle Contributions\T\B}\\
\hline
\hline
free point at 0 & $\displaystyle\frac{t}{d}$ & $\displaystyle\frac{t^{l-1}\nu_j\T}{\prod_{i=1}^k\nu_i^{m_i}\B}$\\  \hline 
smooth node at 0 & // & $\displaystyle\frac{t\T}{\frac{t}{\nu_j\B}-\psi_N}$ \\  \hline 
smooth node at $\infty$ & $\displaystyle\frac{1}{-t-\psi}$ & $\displaystyle\frac{1\T}{-t-\psi\B}$\\  \hline 
edge & $\displaystyle\frac{d^d}{d!t^d}$& $\displaystyle\frac{1}{t^d}\prod_{i=1\B}^{k\T}\left(\frac{\nu_i^{\nu_i}}{\nu_i!}\right)^{m_i}$\\  \hline 
Hodge & // & $\displaystyle\frac{\T t^g-t^{g-1}\lambda_1+\cdots+(-1)^g\lambda_g}{\B t}$\\  \hline 
$\br^\ast([H]^{d-1})$ & $(d-1)!t^{d-1}$ & $\displaystyle\frac{\T (2g+d-1-(l-1))!}{\B (2g-(l-1))!}t^{d-1}$ for $l\leq 2g+1$ \\ 
&& $\T\B 0$ for $l>2g+1 $ \\  \hline
\multicolumn{3}{c}{}
\end{tabular}
\caption{Localization Data for Theorem \ref{theorem:T=2}}
\label{T2data}
\end{table}

\subsection{Proof of Theorem \ref{theorem:T=2}}
The stabilization map: 
$$
\mu: \rma{g}{d} \to \cms_{g,1}
$$
is $\C^\ast$-equivariant with respect to the natural action on the space of maps, and the trivial action on the moduli space of curves. Let ${[\sK]} \in \Chow^\ast_{\C^\ast}(\rma{g}{d})$ be an equivariant lift of  the class $\text{br}^\ast([H]^{d-1})\lambda_g\lambda_{g-1}[1]^{vir}$. Then  applying Formula (\ref{gpvl}) and then pushing forward via $\mu$:

\begin{equation}
\label{auxT2}
\mu_\ast({[\sK]})= \sum_i \mu_{|F_i\ast}\left(\frac{{[\sK]}_{|F_i}}{e(N^{vir}_{F_i})}\right)
\end{equation}

\noindent Since the left hand side of (\ref{auxT2}) is a polynomial in the equivariant parameter $t$, the coefficient of $\frac{1}{t}$ on the right hand side must vanish. We evaluate such relations.\\
{\bf Note:} we (further) abuse notation by omitting $\mu_{|F_i\ast}$.

We choose the equivariant lift of $\text{br}^\ast([H]^{d-1})$ requiring a minimum of $d-1$ ramification above $0$.  Thus the only localization graph with a genus $g$ contracting component above $\infty$ has a single edge representing a trivial degree $d$ cover of $\PP^1$, $
\Gamma_R(d)$ (first column of Table \ref{T2data}).  The other contributing graphs $\Gamma_L(\nu;j)$ are parametrized by a partition $\nu=(\nu_1^{m_1}\cdots\nu_k^{m_k})\vdash d$  and an integer $j\in\{1,\ldots,k\}$.  Let $l$ be the length of the partition.  The $\Gamma_L(\nu;j)$ are found in the second column of Table \ref{T2data}.

The total contribution from Column 1 is:
\begin{align*}
&\frac{t}{d}\left(\frac{1}{-t-\psi}\right)\frac{d^d}{d!t^d}(d-1)!t^{d-1}\lambda_g\lambda_{g-1}\virclass{\abrub{g}{d}{d}}\\
=&\frac{1}{t}\left(\frac{-d^{d-1}(d-1)!}{d!}\lambda_g\lambda_{g-1}\virclass{\abrub{g}{d}{d}}\right)
\end{align*}

The total contribution from the graph $\Gamma(\nu,j)$ (from Column 2 when $l\leq 2g+1$) is:
\begin{align*}
&m_j\left(\frac{t^{l-1}\nu_j}{(\prod_{i=1}^k\nu_i^{m_i})(\prod_{i=1}^k m_i!)}\right)\left(\frac{t}{\frac{t}{\nu_j}-\psi_N}\right)\left(\frac{1}{-t-\psi}\right)\left(\displaystyle\frac{1}{t^d}\prod_{i=1}^k\left(\frac{\nu_i^{\nu_i}}{\nu_i!}\right)^{m_i}\right)\\&\left(\frac{t^g-t^{g-1}\lambda_1+\cdots+(-1)^g\lambda_g}{t}\right)\left(\frac{(2g+d-l)!}{(2g-l+1)!}\right)t^{d-1}\lambda_g\lambda_{g-1}[\cms_{g,1}]\times\virclass{\abrub{0}{\nu}{d}}.
\end{align*}
Expanding terms, this becomes 
\begin{align*}
&-t^{l-4}m_j\left(\frac{\nu_j^2}{\prod_{i=1}^k m_i!}\right)\left(\prod_{i=1}^k\left(\frac{\nu_i^{\nu_i-1}}{\nu_i!}\right)^{m_i}\right)\frac{(2g+d-l)!}{(2g-l+1)!}\left(1+\frac{\nu_j\psi_N}{t}+\frac{\nu_j^2\psi_N^2}{t^2}+\cdots\right)\left(1-\frac{\psi}{t}+\frac{\psi^2}{t^2}\mp\cdots\right)\\&\left(t^g-t^{g-1}\lambda_1+\cdots+(-1)^g\lambda_g\right)\lambda_g\lambda_{g-1}[\cms_{g,1}]\times\virclass{\abrub{0}{\nu}{d}}.
\end{align*}

Taking the $\frac{1}{t}$ coefficient, from Column 1 we get
\begin{align*}
\frac{-d^{d-1}(d-1)!}{d!}\lambda_g\lambda_{g-1}\virclass{\abrub{g}{d}{d}},
\end{align*}
and from Column 2:
\begin{align*}
&(-1)^{l-3}m_j\left(\frac{\nu_j^2}{\prod_{i=1}^k m_i!}\right)\left(\prod_{i=1}^k\left(\frac{\nu_i^{\nu_i-1}}{\nu_i!}\right)^{m_i}\right)\frac{(2g+d-l)!}{(2g-l+1)!}\\ &\psi^{l-2}\left(\nu_j^{g-1}\psi_N^{g-1}-\nu_j^{g-2}\psi_N^{g-2}\lambda_1+\cdots+(-1)^{g-1}\lambda_{g-1}\right) \lambda_g\lambda_{g-1}[\cms_{g,1}]\times\virclass{\abrub{0}{\nu}{d}}.
\end{align*}

Genus $0$, one part, marked double Hurwitz numbers are evaluated in \cite{gjv:dhn}:
\begin{equation}\label{gzop}
\text{H}^0_{\nu,(d)}=(l-1)!d^{l-2}.
\end{equation}
Using \eqref{gzop} and Lemma \ref{lem:psiclass}, the class $\psi^{l-2}$ evaluated on $\abrub{0}{\nu}{d}$ is $d^{l-2}$.  We  simplify notation and get rid of the $m_i$'s and $m_j$'s by writing our partition $\nu$ additively in all it's parts $\nu_1+\cdots+ \nu_l$ and writing $\text{Aut}(\nu)$ for $\prod_{i=1}^k m_i!$.  Summing all terms together and solving for $\lambda_g\lambda_{g-1}\virclass{\abrub{g}{d}{d}}$, we get:

\begin{align}
\label{vvv}
& \lambda_g\lambda_{g-1}\virclass{\abrub{g}{d}{d}}\\
=&\sum_{\nu\vdash d}(-1)^{l-3}\frac{d!}{d^{d-1}}\binom{2g+d-l}{d-1}\frac{d^{l-2}}{\text{Aut}(\nu)}\sum_{j=1}^l\nu_j^2\left(\prod_{i=1}^l\frac{\nu_i^{\nu_i-1}}{\nu_i!}\right)\int_{[\cms_{g,1}]}\frac{1-\lambda_1+\cdots+(-1)^g\lambda_g}{1-\nu_j\psi_1}\lambda_g\lambda_{g-1}[\text{pt}] \nonumber
\label{eqn:bypart}
\end{align}

\noindent The summation is over only those partitions of length $l\leq 2g+1$ since the graph contribution is zero otherwise.  

We borrow some notation from \cite{FP00} in anticipation of using one of their results.  Let $Q_g^e$ be the Hodge integral
\begin{align*}
\int_{[\cms_{g,1}]}\frac{1-\lambda_1+\cdots+(-1)^g\lambda_g}{1-e\psi_1}\lambda_g\lambda_{g-1}
\end{align*}
appearing in \eqref{vvv}.  For any formal power series $f(x)=\sum_if_ix^i$, let $\sC(x^i,f(x))=f_i$ denote the coefficient of $x^i$.  We  use $\sC(x^{d-e},\tau^l(x))$  for $\tau(x)=\sum_{r\geq1}\frac{r^{r-1}}{r!}$ to collect some of the coefficients in our expression.

Summing over possible values of $Q_g^e$, our expression becomes:
\begin{align*}
&\frac{1}{d^{d-1}}\sum_{e=1}^dQ_g^e\sum_{\substack{\nu\vdash d:\\ e\in\nu}}(-1)^{l-3}d!\binom{2g+d-l}{d-1}\frac{d^{l-2}}{\text{Aut}(\nu)}m_ee^2\left(\prod_{i=1}^l\frac{\nu_i^{\nu_i-1}}{\nu_i!}\right)[\text{pt}]
\end{align*}
where $m_e\in \Z$ is the number of parts in $\nu$ of size $e$.  Summing instead over partitions of $d-e$, we get:
\begin{align*}
&\frac{1}{d^{d-1}}\sum_{e=1}^dQ_g^e\sum_{\nu\vdash (d-e)}\frac{(2g+d-l-1)!}{(2g-l)!}\frac{(-d)^{l}}{\text{Aut}(\nu)}\frac{e^{e+1}}{e!}\left(\prod_{i=1}^l\frac{\nu_i^{\nu_i-1}}{\nu_i!}\right)[\text{pt}].
\end{align*}

\noindent Now the summation is over only the partitions of length $l\leq 2g$.

Notice that 
\[
\sC(x^{d-e},\tau^l(x))=\sum_{\substack{\nu\vdash (d-e):\\ \text{length}(\nu)=l}} \frac{l!}{\text{Aut}(\nu)}\left(\prod_{i=1}^l\frac{\nu_i^{\nu_i-1}}{\nu_i!}\right),
\]
so first summing over the possible lengths of a partition, our expression simplifies to
\begin{align*}
\lambda_g\lambda_{g-1}\virclass{\abrub{g}{d}{d}}=\frac{1}{d^{d-1}}\sum_{e=1}^dQ_g^e\frac{e^{e+1}}{e!}\sum_{l=1}^{2g}\frac{(2g+d-l-1)!}{(2g-l)!}\frac{(-d)^{l}}{l!}\sC(x^{d-e},\tau^l(x))[\text{pt}].
\end{align*}
These numbers fit into a generating function
\[
G(y)=\sum_{g=1}^\infty\frac{1}{d^{d-1}}\sum_{e=1}^dQ_g^e\frac{e^{e+1}}{e!}\sum_{l=1}^{2g}\frac{(2g+d-l-1)!}{(2g-l)!}\frac{(-d)^{l}}{l!}\sC(x^{d-e},\tau^l(x))y^{2g}.
\]
This is the same generating function that appears in Proposition~1 of \cite{FP00}, which states that
\[
G(y)=\log\left(\frac{dy/2}{\sin(dy/2)}\right).
\]
Setting $d=1$ gives our series for  $(\lambda_g\lambda_{g-1}\virclass{\abrub{g}{1}{1}})$.  Since the coefficients in the series for general $d$ differ only by a factor of $d^{2g}$, the result follows.\qed
%
%
\section{Total Length 3: Genus 1}

In this section we compute $P_1(d;\alpha_2+\alpha_3)= A_2\alpha_2^2+A_3\alpha_3^2+B\alpha_2\alpha_3$ (Theorem \ref{theorem:T=3}).  We obtain the coefficients of $\alpha_2^2, \alpha_3^2$ terms by pullback from the coefficients of theorem \ref{theorem:T=2}. Computing the coefficient of $\alpha_2\alpha_3$ requires an auxiliary localization computation.

%
%
\subsection{Computing $A_2, A_3$.}

Let $\pi_3: \cms_{1,3}\to \cms_{1,2}$ denote the morphism forgetting the third marked point.
By Theorem \ref{thm:Qpoly}, $A_2\alpha_2^2=P_1(d;\alpha_2+0)= \pi_3^\ast(P_1(\alpha_2;\alpha_2))$. The length two polynomial is computed in Theorem \ref{theorem:T=2}, and 
$$
\lambda_1 P_1(\alpha_2;\alpha_2)= \frac{1}{24}[pt.]\alpha_2^2. 
$$
It follows that
\begin{equation} \label{A2}
A_2=\pi_3^\ast(\overline{\text{D}}_{1,0}(\emptyset|1,2))= 
\overline{\text{D}}_{1,0}(\emptyset|1,2,3)+\overline{\text{D}}_{1,0}(3|1,2)= \psi_1- \overline{\text{D}}_{1,0}(2|1,3).
\end{equation}
The last equality in (\ref{A2}) is a simple relation in the tautological ring. On $\cms_{1,1}$ we have $\lambda_1=\psi_1$.  Using the comparison lemma to pull-back the class $\psi_1$, one pulls this equality back to $\cms_{1,3}$ and obtains:
\[
\lambda_1 = \psi_1 - \overline{\text{D}}_{1,0}(3|1,2) - \overline{\text{D}}_{1,0}(2|1,3) - \overline{\text{D}}_{1,0}(\emptyset|1,2,3).
\]
Multiplying  by $\lambda_1$ and recalling that $\lambda_1^2=0$,  we get the relation:
\begin{equation}\label{finrel}
0 = \psi_1\lambda_1 - \lambda_1\overline{\text{D}}_{1,0}(3|1,2) - \lambda_1\overline{\text{D}}_{1,0}(2|1,3) - \lambda_1\overline{\text{D}}_{1,0}(\emptyset|1,2,3).
\end{equation}
Since multiplication by $\lambda_1$ gives an isomorphism between $\sR^1(\sM^{rt}_{1,3})$ and $\sR^1(\cms_{1,3})$, relation (\ref{finrel}) holds without the factor of $\lambda_1$ on rational tails.

The coefficient $A_3$ is obtained either by  repeating the same argument using the forgetful map $\pi_2$ or by exploiting the equivariance of the polynomial with respect to the automorphism of $\cms_{1,3}$ exchanging the second and the third mark. 
%
\subsection{Computing $B$.}

For any positive integer $d \geq 2$,  consider the equivariant cohomology class ${[\sK(d)]}\in \Chow^\ast_{\C^\ast}(\rma{1,\{2,3\}}{d})$ which lifts the class
$$\lambda_1\left(\prod_{i=2}^{3} ev_i^{\ast}(pt.)\right) br^\ast([H]^{d-2})$$ by requiring the marked points to map to $0$ and also at least $d-2$ ramification over $0$.
Applying  (\ref{gpvl}) and pushing forward via the stabilization morphism  $\mu:\cms_{1,\{2,3\}}(d)\to\cms_{1,3}$ (where the fully ramified point is remembered as the first mark), we obtain a description of $\mu_\ast({[\sK(d)]})$ in terms of rational functions in the equivariant parameter $t$ with coefficients in the cohomology of the (pushforwards of) fixed loci of the moduli space of maps (as in (\ref{auxT2})). We denote by $R(d)$ the $\frac{1}{t}$ coefficient of the localization of ${[\sK(d)]}$ and observe that $\mu_\ast(R(d))=0 \in R^\ast(\cms_{1,3})$. We carry out this computation explicitly for $d=2$. We introduce some more notation for the localization graphs: genus is in the superscript and marked points are in the subscript of the corresponding part of the partition $\nu$.  
Table \ref{table:deg2} lists the localization data.

\begin{table}
\begin{tabular}[c]{| l | c | l | l |} 
\hline
Graph &  $\frac{\T m(\Gamma)}{\B |\A_\Gamma|}$& Fixed Locus & Normal Contribution\\ \hline \hline
1. $\T\Gamma_L(2_{2,3}^1)$ & $\frac{1}{2}$ & $\cms_{1,3}$ & $\frac{ t}{\frac{t}{2}-\psi_N}\frac{2}{t^2}\frac{t-\lambda_1}{t}t^2$ \\
2. $\Gamma_R(2_{2,3})$ & 1 & $\cms_{0,3}\times\abrub{1}{2}{2}$& $\frac{t}{\frac{t}{2}-\psi_N'}\frac{1}{-t-\psi}\frac{2}{t^2}\frac{1}{t}t^2$\\
3. $\Gamma_L(1_{2,3}^1+1)$ & 1 & $\cms_{1,3}\times\abrub{0}{2}{1+1}$ & $\frac{t}{t-\psi_N}\frac{1}{-t-\psi}\frac{1}{t^2}\frac{t-\lambda_1}{t}t^3$\\
4. $\Gamma_L(1_2^1+1_3)$, $\Gamma_L(1_3^1+1_2)$ & 1 & $2\left(\cms_{1,2}\times\abrub{0}{2}{1+1}\right)$ &$\frac{t}{t-\psi_N}\frac{1}{-t-\psi}\frac{1}{t^2}\frac{t-\lambda_1}{t}t^2$\\
5. $\Gamma_L(1^1+1_{2,3})$ & 1 & $\cms_{0,3}\times\cms_{1,1}\times\abrub{0}{2}{1+1}$ & $\frac{t}{t-\psi_N'}\frac{t}{t-\psi_N}\frac{1}{-t-\psi}\frac{1}{t^2}\frac{t-\lambda_1}{t^2}t^2$\\
6. $\Gamma_R(1_{2,3}+1)$ & 1 & $\cms_{0,3}\times\abrub{1}{2}{1+1}$& $\frac{t}{t-\psi_N'}\frac{1}{-t-\psi}\frac{1}{t^2}\frac{1}{t}t^3$\\
7. $\B\Gamma_R(1_2+1_3)$ & 1 & $\abrub{1}{2}{1+1}$ &$\frac{1}{-t-\psi\B}\frac{1}{t^2}t^2$\\
\hline
\end{tabular}
\vspace{0.3cm}
\caption{Degree 2 Localization Data}
\label{table:deg2}
\end{table}

\emph{Degree 2}: 
\begin{align*}
R(2) =&4\psi_N\lambda_1[\cms_{1,3}]-\psi_N\lambda_1[\cms_{1,3}]\times[\abrub{0}{2}{1+1}] - 2\lambda_1[\cms_{1,2}]\times[\abrub{0}{2}{1+1}]\\
-&\lambda_1[\cms_{0,3}]\times[\abrub{1}{2}{1+1}] - \lambda_1[\abrub{1}{2}{1+1}].
\end{align*}
The contributions arise from graphs 1, 3, 4, 6, and 7.  The graphs 2 and 6 do not contribute. 

Next we pushforward  $R(2)$ and obtain a relation on the moduli space of curves which we call $L(2)$. The genus zero class $[\abrub{0}{2}{1+1}]$ pushes forward to the Hurwitz number $\text{H}^0_{2,1+1}=1$ times the class of a point. We substitute the appropriate boundary expression for the genus $0$ rubber classes:
\[
L(2):=\mu_\ast(R(2))= 4\psi_1\lambda_1- \psi_1\lambda_1 - \lambda_1\overline{\text{D}}_{1,0}(2|1,3)-\lambda_1\overline{\text{D}}_{1,0}(3|1,2)-\lambda_1\text{S}(1_{2,3}+1) - \lambda_1\mu_\ast[\abrub{1}{2}{1+1}]=0.
\]
This simplifies to
\begin{equation}\label{deg2prepush}
3\psi_1\lambda_1 -\lambda_1\overline{\text{D}}_{1,0}(2|1,3)-\lambda_1\overline{\text{D}}_{1,0}(3|1,2)=\lambda_1\text{S}(1_{2,3}+1) + \lambda_1\mu_\ast[\abrub{1}{2}{1+1}].
\end{equation}
The pushforward of $R(2)$ still contains classes that are neither standard generators nor the pushforward  of the rubber classes we are interested in. We next express such classes in terms of standard classes.

\noindent{\emph{The $S$-Loci}}

\begin{figure}
	\centering
		\includegraphics[width=0.55\textwidth]{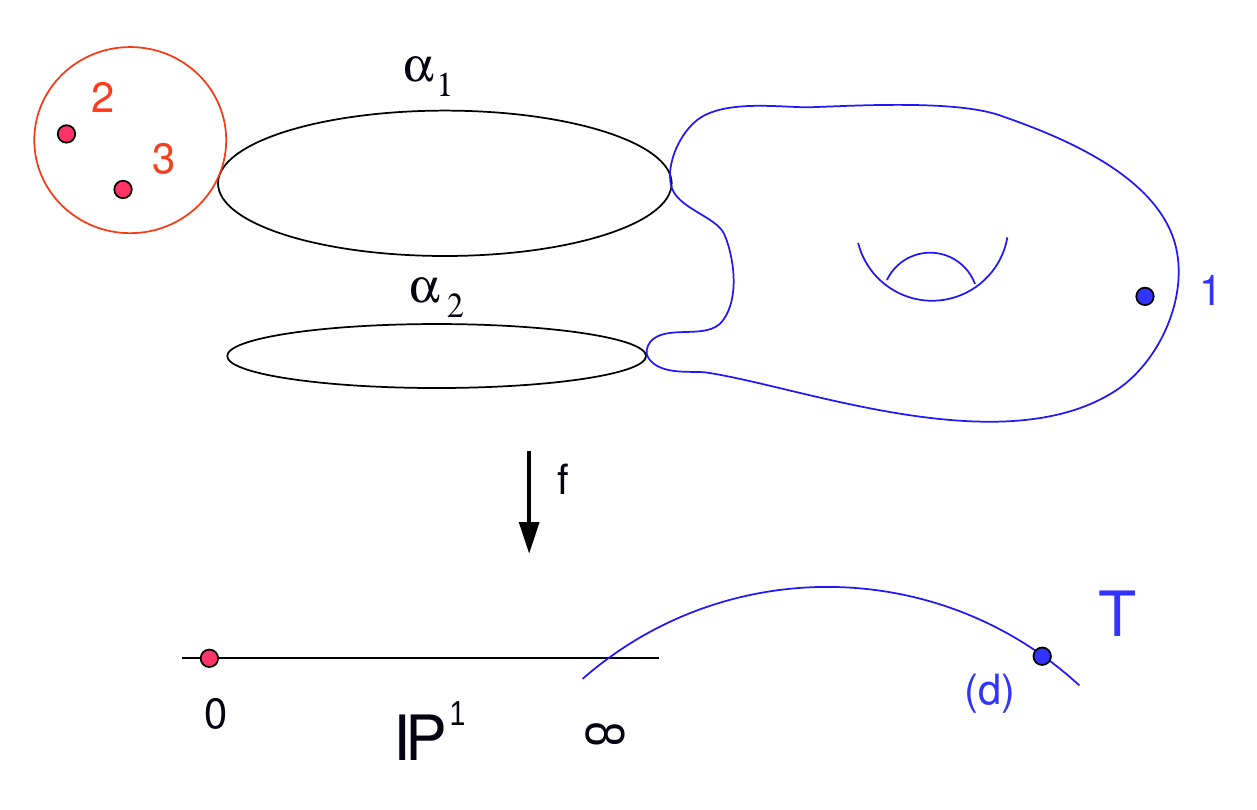}
	\caption{The fixed locus in $\rma{1,\{2,3\}}{d}$ that pushes forward to $S((\alpha_2)_{2,3}+\alpha_3)$}
	\label{fig:Slocus}
\end{figure}

We denote by $S((\alpha_2)_{2,3}+\alpha_3)$  the push forward through $\mu$ of the class $[\cms_{0,3}]\times[\abrub{1}{d}{\alpha_2+\alpha_3}]$ given by the fixed locus corresponding to the graph $\Gamma_R((\alpha_2)_{2,3}+\alpha_3)$, and similarly $S({\alpha_2}+(\alpha_3)_{2,3})$(Figure \ref{fig:Slocus}).    For $i=1,2,3$, let $\pi_i:\cms_{1,3}\to\cms_{1,2}$ be the forgetful morphism and $\sigma_i:\cms_{1,2}\to\cms_{1,3}$  the section, each corresponding to the $i$-th marked point.  

Note we have:
\begin{align}
\lambda_1 S((\alpha_2)_{2,3}+{\alpha_3})=\lambda_1{\sigma_2}_\ast{\pi_3}_\ast\mu_\ast[\abrub{1}{d}{\alpha_2+\alpha_3}]= \lambda_1(A_2'\alpha_2^2 +A_3'\alpha_3^2 +B'\alpha_2\alpha_3).
\end{align}
From (\ref{A2}) we can compute:
\begin{align}
A_2' &= {\sigma_2}_\ast{\pi_3}_\ast \pi_3^\ast(\overline{\text{D}}_{1,0}(\emptyset|1,2))= 0 \nonumber \\
A_3' &= {\sigma_2}_\ast{\pi_3}_\ast \pi_2^\ast(\overline{\text{D}}_{1,0}(\emptyset|1,3))= \overline{\text{D}}_{1,0}(1|2,3).
\end{align}
To determine $B'$ we ``push-push" $L(2)$:

${\pi_3}_\ast L(2)$:
\begin{align*}
2\lambda_1= \lambda_1{\pi_3}_\ast\text{S}(1_{2,3}+1)+\lambda_1{\pi_3}_\ast\mu_\ast[\abrub{1}{2}{1+1}];
\end{align*}
${\sigma_2}_\ast{\pi_3}_\ast L(2)$:
\begin{align*}
2\lambda_1\overline{\text{D}}_{1,0}(1|2,3)= 2\lambda_1\text{S}(1_{2,3}+1).
\end{align*}
This gives the linear equation in the coefficients of $S((\alpha_2)_{2,3}+\alpha_3)$:
\begin{align}\label{pushpush2}
\overline{\text{D}}_{1,0}(1|2,3)=\overline{\text{D}}_{1,0}(1|2,3)+B',
\end{align}
determining $B'=0$. Now we can use $L(2)$ to solve for $B$. Plugging in (\ref{deg2prepush}) the boundary expressions for $A_2,A_3$ and $\text{S}(1_{2,3}+1)$, we obtain the linear equation on the coefficients:
\begin{align}\label{almostthere2}
3\psi_1\lambda_1 -\overline{\text{D}}_{1,0}(2|1,3)-\overline{\text{D}}_{1,0}(3|1,2) = \overline{\text{D}}_{1,0}(1|2,3)+\psi_1-\overline{\text{D}}_{1,0}(2|1,3) +\psi_1-\overline{\text{D}}_{1,0}(3|1,2)+B,
\end{align}
which determines:
\begin{equation}
B=\psi_1 - \overline{\text{D}}_{1,0}(1|2,3).
\end{equation}

\subsection{Consistency check: degree $3$ relation.}

In this section we present an auxilary virtual localization relation among our classes. This serves the twofold purpose of giving a consistency check of the previous computations and of providing the reader that may be interested in applying this techniques with another example of it. We check that the relation $L(3):= \mu_\ast R(3)$ is compatible with the result of theorem \ref{theorem:T=3}. The localization data is contained in table \ref{table:deg3}.

\begin{table}
\begin{adjustwidth}{-0.5in}{-1in}
\begin{tabular}[c]{| l | c | l | l |}
\hline
Graph &  $\frac{\T m(\Gamma)}{\B |\A_\Gamma|}$ & Fixed Locus& Normal Contribution \\ \hline \hline
1. \T $\Gamma_L(3_{2,3}^1)$ & $\frac{1}{3}$ &$\cms_{1,3}$& $\frac{t}{\frac{t}{3}-\psi_N}\frac{3^3}{3!t^3}\frac{t-\lambda_1}{t}4t^3$\\
2. $\Gamma_R(3_{2,3})$ & 1&$\cms_{0,3}\times\abrub{1}{3}{3}$&$\frac{t}{\frac{t}{3}-\psi_N'}\frac{1}{-t-\psi}\frac{3^3}{3!t^3}\frac{1}{t}2t^3$\\
3. $\Gamma_L(2_{2,3}^1+1)$ &1 &$\cms_{1,3}\times\abrub{0}{3}{2+1}$&$\frac{t}{\frac{t}{2}-\psi_N}\frac{1}{-t-\psi}\frac{2}{t^3}\frac{t-\lambda_1}{t}3t^4$ \\
4.  $\Gamma_L(2_2^1+1_3)$, $\Gamma_L(2_3^1+1_2)$ & 1 & $2\left(\cms_{1,2}\times\abrub{0}{3}{2+1}\right)$& $\frac{t}{\frac{t}{2}-\psi_N}\frac{1}{-t-\psi}\frac{2}{t^3}\frac{t-\lambda_1}{t}3t^3$\\
5. $\Gamma_L(2^1+1_{2,3})$ & 1&$\cms_{0,3}\times\cms_{1,1}\times\abrub{0}{3}{2+1}$& $\frac{t}{\frac{t}{2}-\psi_N}\frac{t}{t-\psi_N'}\frac{1}{-t-\psi}\frac{2}{t^3}\frac{t-\lambda_1}{t^2}3t^3$\\
6. $\Gamma_L(2_{2,3}+1^1)$ & 1 &$\cms_{0,3}\times\cms_{1,1}\times\abrub{0}{3}{2+1}$&$\frac{t}{\frac{t}{2}-\psi_N'}\frac{t}{t-\psi_N}\frac{1}{-t-\psi}\frac{2}{t^3}\frac{t-\lambda_1}{t^2}3t^3$\\
7. $\Gamma_L(2_2+1_3^1)$, $\Gamma_L(2_3+1_2^1)$ & 1 &$2\left(\cms_{1,2}\times\abrub{0}{3}{2+1}\right)$& $\frac{t}{t-\psi_N}\frac{1}{-t-\psi}\frac{2}{t^3}\frac{t-\lambda_1}{t}3t^3$ \\
8. $\Gamma_L(2+1^1_{2,3})$ & 1 &$\cms_{1,3}\times\abrub{0}{3}{2+1}$& $\frac{t}{t-\psi_N}\frac{1}{-t-\psi}\frac{2}{t^3}\frac{t-\lambda_1}{t}\frac{3}{2}t^4$ \\
9. $\Gamma_R(2_{2,3}+1)$ & 1 &$\cms_{0,3}\times\abrub{1}{3}{2+1}$& $\frac{t}{\frac{t}{2}-\psi_N'}\frac{1}{-t-\psi}\frac{2}{t^3}\frac{1}{t}t^3$\\ 
10. $\Gamma_R(2_2+1_3)$& 1 & $\abrub{1}{3}{2_2+1_3}$& $\frac{1}{-t-\psi}\frac{2}{t^3}t^3$\\
11. $\Gamma_R(2_3+1_2)$& 1 & $\abrub{1}{3}{2_3+1_2}$& $\frac{1}{-t-\psi}\frac{2}{t^3}t^3$\\
12. $\Gamma_R(2+1_{2,3})$ & 1 &$\cms_{0,3}\times\abrub{1}{3}{2+1}$& $\frac{t}{t-\psi_N'}\frac{1}{-t-\psi}\frac{2}{t^3}\frac{1}{t}\frac{1}{2}t^4$\\
13. $\Gamma_L(1^1_{2,3}+1+1)$ & $\frac{1}{2}$&$\cms_{1,3}\times\abrub{0}{3}{1+1+1}$&$\frac{t}{t-\psi_N}\frac{1}{-t-\psi}\frac{1}{t^3}\frac{t-\lambda_1}{t}2t^5$\\ 
14. $\Gamma_L(1^1_2+1_3+1)$, $\Gamma_L(1^1_3+1_2+1)$ & 1 &$2\left(\cms_{1,2}\times\abrub{0}{3}{1+1+1}\right)$&$\frac{t}{t-\psi_N}\frac{1}{-t-\psi}\frac{1}{t^3}\frac{t-\lambda_1}{t}2t^4$\\ 
15. $\Gamma_L(1^1+1_{2,3}+1)$ & 1 &$\cms_{0,3}\times\cms_{1,1}\times\abrub{0}{3}{1+1+1}$&$\frac{t}{t-\psi_N'}\frac{t}{t-\psi_N}\frac{1}{-t-\psi}\frac{1}{t^3}\frac{t-\lambda_1}{t^2}2t^4$\\ 
16. \B$\Gamma_L(1^1+1_2+1_3)$ & 1&$\cms_{1,1}\times\abrub{0}{3}{1+1+1}$&$\frac{t}{t-\psi_N}\frac{1}{-t-\psi\B}\frac{1}{t^3}\frac{t-\lambda_1}{t}2t^3$\\ 
\hline
\end{tabular}
\end{adjustwidth}
\vspace{0.3cm}
\caption{Degree 3 Localization Data: Fixed Locus Information}
\label{table:deg3}
\end{table}

\begin{align*}
L(3)=0= &54\psi_N\lambda_1[\cms_{1,3}] - 24\psi_N\lambda_1[\cms_{1,3}]\times[\abrub{0}{3}{2+1}]-24\lambda_1[\cms_{1,2}]\times[\abrub{0}{3}{2+1}] \\
-&12\lambda_1[\cms_{1,2}]\times[\abrub{0}{3}{2+1}]-3\psi_N\lambda_1[\cms_{1,3}]\times[\abrub{0}{3}{2+1}]-4\lambda_1[\cms_{0,3}]\times[\abrub{1}{3}{2+1}]\\
-&2\lambda_1[\abrub{1}{3}{2_2+1_3}]-2\lambda_1[\abrub{1}{3}{2_3+1_2}]-\lambda_1[\cms_{0,3}]\times[\abrub{1}{3}{2+1}]\\
+&\psi_N\psi\lambda_1[\cms_{1,3}]\times[\abrub{0}{3}{1+1+1}]
+4\psi\lambda_1[\cms_{1,2}]\times[\abrub{0}{3}{1+1+1}]\\-&2\lambda_1[\cms_{0,3}]\times[\cms_{1,1}]\times[\abrub{0}{3}{1+1+1}]
-2\lambda_1[\cms_{1,1}]\times[\abrub{0}{3}{1+1+1}]
\end{align*}
Graphs 1, 3, 4, and 7 through 16 contribute to the above expression.  The graphs 2, 5, and 6 do not contribute. Note that in graphs 10 and 11 we carefully keep track of the location of the marks. While in principle we should do this for every single fixed locus, we allow ourselves the sloppiness of forgetting such information from our notation when it is irrelevant to the computation.

{\emph{Genus $0$ Rubber Classes.}} 

The genus zero rubber classes contribute $[\abrub{0}{3}{2+1}]=1[\text{pt}]$ and $\psi[\abrub{0}{3}{1+1+1}]=2[\text{pt}]$ for the terms corresponding to the fixed loci 3, 4, 7, 8, 12 and 13.  The term $2\lambda_1[\cms_{0,3}]\times[\cms_{1,1}]\times[\abrub{0}{3}{1+1+1}]$ corresponding to locus 14 pushes forward to 0.    The last term, containing $[\abrub{0}{3}{1+1+1}]$ and no $\psi$ class, 
does not push forward to zero: the $\abrub{0}{3}{1+1+1}$ factor 
maps isomorphically onto the $\PP^1(=\overline{M}_{0,4})$ component of $\overline{\text{D}}_{1,0}(\emptyset|1,2,3)$.  This locus thus pushes forward to $-2\lambda_1\overline{\text{D}}_{1,0}(\emptyset|1,2,3)$.

Substituting the boundary expressions for the genus $0$ rubber classes into $L(3)$ and simplifying, one obtains:

\begin{align}\label{almostthere3}
2\lambda_1\mu_\ast[\abrub{1}{3}{2_2+1_3}]+2\lambda_1\mu_\ast[\abrub{1}{3}{2_3+1_2}] =& 30\psi_1\lambda_1 - 12\lambda_1\overline{\text{D}}_{1,0}(2|1,3) - 12\lambda_1\overline{\text{D}}_{1,0}(3|1,2)  \nonumber \\
& \hspace{-2cm}- 2\lambda_1\overline{\text{D}}_{1,0}(\emptyset|1,2,3)- 4\lambda_1\text{S}(2_{2,3}+1)-\lambda_1\text{S}(2+1_{2,3})\nonumber \\
\end{align}

Using the results of the previous two subsections and forgetting $\lambda_1$, (\ref{almostthere3}) becomes:

\begin{align}\label{there3}
10(\psi_1- \overline{\text{D}}_{1,0}(2|1,3))+10(\psi_1- \overline{\text{D}}_{1,0}(3|1,2))+8 (\psi_1- \overline{\text{D}}_{1,0}(1|2,3))  = \nonumber \\ 30\psi_1 - 12\overline{\text{D}}_{1,0}(2|1,3) - 12\overline{\text{D}}_{1,0}(3|1,2) - 2\overline{\text{D}}_{1,0}(\emptyset|1,2,3)  - 6\overline{\text{D}}_{1,0}(1|2,3)\nonumber, 
\end{align}
which immediately simplifies to relation (\ref{finrel}).

\subsection{Genus 1 polynomial for arbitrary total length: Corollary \ref{cor:arbitraryT}}

By Corollary \ref{cor:pullback} our computation of $P_{1,3}(d;\alpha_2,\alpha_3)$ determines the genus 1 polynomial $P_{1,T}(d;\alpha_2,\ldots,\alpha_T)$ for arbitrary total length.  This is computation is performed by applying parts (i) and (ii) of Theorem \ref{thm:Qpoly}.  First, pull back the coefficients of $\alpha_2^2$ and $\alpha_2\alpha_3$ from $P_{1,3}(d;\alpha_2,\alpha_3)$ iteratively through maps forgetting marked points.  Then, apply the equivariance from part (i).  The resulting polynomial is as stated in Corollary \ref{cor:arbitraryT}.

\appendix

\section{Obstruction theory details} \label{app:obs}

In this appendix, we will check that the definitions of Section~\ref{sec:OT} do indeed yield obstruction theories and that these obstruction theories are perfect.

\subsection{Calculations} \label{sec:calc}

In this section we develop some explicit descriptions of the sheaves $\uT$ and $\uT'$.  These descriptions will not be needed directly in the sequel, but we will need to know that a large part of $\uT(CS,J)$ and $\uT'(CS,J)$ can be constructed from quasi-coherent sheaves on $C$ and on $S$.  The techniques used in this section are very similar for $\fE$ and $\fE'$.  We will generally only give the arguments in detail for $\fE$ and then explain what modifications are necessary when repeating them for $\fE'$.  


Let $C$ be the family of curves over $S$ associated to an $S$-point of $\oM_{\rel}(\sP/\BGm)$.  We begin by relating the abelian group stack $\uT$ on $CS$, defined in Section~\ref{sec:obthy1}, to some more familiar sheaves.  Our methods here are adapted from \cite[Section~5]{L02}.

\subsubsection{The sheaf $\uA$}

Suppose that the map $S \rightarrow \tsT$ induced from an $S$-point of $\oM_{\rel}(\sP/\BGm)$ factors through a map $S \xrightarrow{f} W$ where $W$ is smooth over $\tsT$.  Let $A(CS,J)$ be the category of predeformable completions of the diagram
\begin{equation*} \xymatrix{
    C \ar[r] \ar[d] \ar@/^15pt/[rr] & C[p^\ast J] \ar@{-->}[r] \ar[d] & \tsP_W \ar[d] \\
    S \ar[r] \ar@/_15pt/[rr] & S[J] \ar@{-->}[r] & W .
  }
\end{equation*}
Extend this, in the usual way, to an abelian group stack $\uA(CS,J)$ on $CS$.

\begin{lem} \label{lem:A}
  Assume that the map $S \rightarrow \tsT$ (resp.\ $S \rightarrow \sT^2$) associated to an $S$-point of $\oM_{\rel}(\sP/\BGm)$ (resp.\ of $\oM_{\rel}(\sP)$) factors through $f : S \rightarrow W$ for some $W$ smooth over $\tsT$ (resp.\ over $\sT^2$).  There is an exact sequence 
  \begin{equation}\label{eqn:23}
    0 \rightarrow \pi^\ast \uT_{W/\tsT}(S,J) \rightarrow \uA(CS,J) \rightarrow \uT(CS,J) \rightarrow 0 
  \end{equation}
  of abelian group stacks on $CS$.
\end{lem}
\begin{proof}
  Write $\sZ$ for $\tsT$ (resp.\ for $\sT^2$).  Since $W$ is smooth over $\sZ$, the map $\uA \rightarrow \uT$ is surjective, by the formal criterion for smoothness.  The sections of the kernel over $UV$ are the completions of the diagram
  \begin{equation*} \xymatrix{
      V \ar[r]^f \ar[d] & W \ar[d] \\
      V[J] \ar[r]^0 \ar@{-->}[ur] & \tsT .
    }
  \end{equation*}
  By inspection, this is the same as $\pi^\ast \uT_{W/\sT}(S,J)$.
\end{proof}

Now let $g : C \rightarrow \tsP$ be the map induced from the $S$-point of $\oM_{\rel}(\sP/\BGm)$.  Recall that $\tsP$ and $\tsT$ can be given log.\ structures, and the map $\tsP \rightarrow \tsT$ can be extended to a log.\ smooth morphism.  Let $\Omega_{\tsP / \tsT}(\log)$ be the sheaf of relative log.\ differentials.

\begin{remark}
  When we consider an $S$-point of $\oM_{\rel}(\sP)$, we will need to use the relative cotangent complex $\Omega_{\tsP / \sT^2}(\log)$, which is perfect in degrees $[0,1]$, since the map $\tsP \rightarrow \sT^2$ is not representable.  This is the only noteworthy difference between the arguments used in this section to study $\uT$ and the analogous ones that apply to $\uT'$.
\end{remark}

It is possible to describe $\uA(CS,J)$ fairly explicitly, following \cite[Sections~1 and~5]{L02}.  Recall that a chart of ``first kind'' of $CS$ is a $UV \in CS$ such that $U$ does not contain any ``essential nodes''---nodes whose images in $\tsP$ meets the singular locus.  A chart of ``second kind'' is a $UV \in CS$ such that $UV$ is a small \'etale neighborhood of an essential node.  Here ``small'' means that the morphism from $U$ to $\tsP$ admits a certain standard description, which we recall below.

Over a chart $UV$ of second kind, the commutative diagram
\begin{equation*} \xymatrix{
    U \ar[r] \ar[d] & \tsP_W \ar[d] \\
    V \ar[r] & W
  }
\end{equation*}
can be obtained by \'etale localization from a commutative diagram of rings,
\begin{equation*} \xymatrix{
      \cO_V[x,y] / (xy - t) & \ar[l] \cO_W[u,v] / (uv - w) \\
      \cO_V \ar[u] & \cO_W \ar[u] \ar[l],
    } 
\end{equation*}
where $w \mapsto t^m$, $u \mapsto x^m$, and $v \mapsto y^m$.  By \cite[Lemma~1.12]{L02}, to extend this to a predeformable diagram
\begin{equation*} \xymatrix{
      \cO_V[x,y] / (xy - t) & (\cO_V + J)[x,y] / (xy - t) \ar[l] & \ar@{-->}[l] \ar@/_15pt/[ll] \cO_W[u,v] / (uv - w) \\
      \cO_V \ar[u] & \cO_V + J \ar[l] \ar[u] & \cO_W \ar[u] \ar@{-->}[l] \ar@/^15pt/[ll]
    } 
\end{equation*}
is the same as to give a commutative diagram
\begin{equation*} \xymatrix{
    J \tensor_{\cO_V} \cO_U & \ar[l] \Omega_{\tsP_W}(\log u, \log v) \\
    J \ar[u] & \ar[l] \Omega_W(\log w) \ar[u] .
  }
\end{equation*}
In \cite[Equation~(1.10)]{L02}, the group of such diagrams is denoted $\Hom(f^\ast \Omega_{\tsP_W/W}, J)^\dagger$, at least when $U$ dominates $V$ (cf.\ the beginning of Section~\ref{sec:extension}).

According to whether we are looking at a chart of first or second kind, we write $\Omega_{\tsP_W}^\dagger$ (resp.\ $\Omega_W^\dagger$) for $\Omega_{\tsP_W}(\log u, \log v)$ or $\Omega_{\tsP_W}$ (resp.\ $\Omega_W(\log w)$ or $\Omega_W$).  These do \emph{not} glue together to form sheaves on $CS$, but the collection of diagrams 
\begin{equation} \label{eqn:20} \xymatrix{
    J \tensor_{\cO_V} \cO_U & \ar[l] \Omega^\dagger_{\tsP_W} \\
    J \ar[u] & \ar[l] \Omega^\dagger_W \ar[u] .
  }
\end{equation}
does form a sheaf on $CS$, which is precisely $\uA(CS,J)$.  

\subsubsection{The sheaf $\uB$}

We continue to assume that there is a factorization of $S \rightarrow \tsT$ through a scheme $W$ that is smooth over $\tsT$.  Let $\uB'(CS,J)$ be the sheaf of abelian groups (or abelian group stack when we are studying $\oM_{\rel}(\sP)$) whose value on $UV \in CS$ is 
\begin{equation*}
  \Hom(g^\ast \Omega_{\tsP / \tsT}(\log)_U, p^\ast (J)_U)
\end{equation*}
where $p : C \rightarrow S$ is the projection.  Note that
\begin{equation*}
  \uB'(CS,J) = i_\ast \ch(\uHom(g^\ast \Omega_{\tsP / \tsT}(\log), p^\ast J)) .
\end{equation*}
We can identify $\uB'(CS,J)$ with a subsheaf of the sheaf whose value on $UV$ is the collection of extensions
\begin{equation}\label{eqn:19} \xymatrix{
    C \ar[r] \ar[d] \ar@/^15pt/[rr]^g & \ar[d] C[p^\ast J] \ar@{-->}[r] & \tsP_W \ar[d] \\
    S \ar[r] \ar@/_15pt/[rr]_f & S[J] \ar[r]^0 & W .
  }
\end{equation}
This is, in turn, a subsheaf of $\uA(CS,J)$, so we get an injective map
\begin{equation*}
  \uB'(CS,J) \rightarrow \uA(CS,J) .
\end{equation*}
We define $\uB(CS,J)$ to be the quotient sheaf.  


Over each chart of first or second kind, there is a diagram
\begin{equation} \label{eqn:24} \xymatrix{
    0 \ar[r] & \Omega_{W}^\dagger \tensor_{\cO_V} \cO_U \ar[r] & \Omega_{\tsP_W}^\dagger \ar[r] & \Omega_{\tsP_W/W}(\log) \ar[r] & 0 \\
    0 \ar[r] & \Omega_W^\dagger \ar[u] \ar[r] & \Omega_{W}^\dagger \ar[r] \ar[u] & 0 \ar[u] \ar[r] & 0
  }
\end{equation}
with exact rows.  The square on the right induces the map of sheaves $\uB'(CS,J) \rightarrow \uA(CS,J)$.  

\begin{lemma} \label{lem:B}
  There is an exact sequence
  \begin{equation} \label{eqn:21}
    0 \rightarrow i_\ast \uHom(g^\ast \Omega_{\tsP / \tsT}(\log), p^\ast J) \rightarrow \uA(CS, J) \rightarrow \uB(CS, J) \rightarrow 0 
  \end{equation}
  and on a chart $UV$ of either first or second kind, we have $\uB(CS, J) = \Hom(\Omega_W^\dagger, J)$.
\end{lemma}
\begin{proof}
  Apply $\Hom(-, \pi^\ast J)$ to the seqence~\eqref{eqn:24} over a chart of either kind.  We get the exact sequence
  \begin{equation*}  
    0 \rightarrow \uHom(i_\ast g^\ast \Omega_{\tsP / \tsT}(\log), \pi^\ast J) \rightarrow \uA(CS, J) \rightarrow \uHom(\pi^\ast \Omega_W^\dagger, \pi^\ast J) \rightarrow \Ext^1(i_\ast g^\ast \Omega_{\tsP / \tsT}(\log), \pi^\ast J)
  \end{equation*}
  Since $i$ is a closed embedding, $i_\ast$ is left adjoint to $i^\ast$.  This implies that
  \begin{equation*}
    \uHom(i_\ast g^\ast \Omega_{\tsP / \tsT}(\log), \pi^\ast J) = \uHom(g^\ast \Omega_{\tsP / \tsT}(\log), p^\ast J) .
  \end{equation*}
  On the other hand 
  \begin{equation*}
    \uHom(\pi^\ast \Omega_W^\dagger, \pi^\ast J) = \uHom(\Omega_W^\dagger, \pi_\ast \pi^\ast J) = \uHom(\Omega_W^\dagger, J) .
  \end{equation*}
  Finally, since $g^\ast \Omega_{\tsP / \tsT}(\log)$ is a vector bundle on $C$, Lemma~\ref{lem:splitting} below implies that $$\Ext^1(i_\ast g^\ast \Omega_{\tsP / \tsT}(\log), \pi^\ast J) =0.$$  Taken together, these facts imply that we have an exact sequence
  \begin{equation*}
    0 \rightarrow i_\ast \uHom(g^\ast \Omega_{\tsP / \tsT}(\log), p^\ast J) \xrightarrow{\alpha} \uA(CS, J) \rightarrow \pi^\ast \uHom(\Omega_W^\dagger, J) \rightarrow 0 .
  \end{equation*}
  As we have defined $\uB(CS, J)$ to be the cokernel of the map $\alpha$ above, we get the desired isomorphism $\uB(CS, J) \simeq \pi^\ast \uHom(\Omega_W^\dagger, J)$.
\end{proof}

To complete the proof of Lemma~\ref{lem:B}, we must prove
\begin{lemma} \label{lem:splitting}
  Suppose that $F$ is a sheaf on $C$ and $J$ is a sheaf on $S$ such that $\uExt^1(F, p^\ast J) = 0$.  Then $\uExt^1(i_\ast F, \pi^\ast J) = 0$.
\end{lemma}
\begin{proof}
  Let
  \begin{equation*}
    0 \rightarrow \pi^\ast J \rightarrow Q \rightarrow i_\ast F \rightarrow 0
  \end{equation*}
  be an extension of $i_\ast F$ by $\pi^\ast J$.  We must show that this sequence splits locally in $CS$.  

Since this is a local problem and $\uExt^1(F, p^\ast J) = 0$, we can assume that the exact sequence
  \begin{equation*}
    0 \rightarrow p^\ast J \rightarrow i^\ast Q \rightarrow F \rightarrow 0
  \end{equation*}
  is split.  The choices of splittings form a torsor under $\Hom(F, p^\ast J)$.  Recall now that $Q$ is determined by giving a map $\alpha : i^\ast Q \rightarrow p^\ast j^\ast Q$.  Choosing an isomorphism $i^\ast Q \simeq F \times p^\ast J$ and noting that the map $p^\ast J \rightarrow p^\ast i^\ast Q = p^\ast J$ is necessarily the identity, we see that the extension $Q$, together with the splitting of $i^\ast Q$, is determined by a map $F \rightarrow p^\ast J$.  The action of $\Hom(F, p^\ast J)$ on the collection of pairs $(Q, i^\ast Q \simeq p^\ast J \times F)$ is by addition on the map $F \rightarrow p^\ast J$.  Choosing the splitting of $i^\ast Q$ correctly, therefore, we can ensure that the map $i^\ast Q \rightarrow p^\ast j^\ast Q$ factors throught the projection of $i^\ast Q$ on $p^\ast J$, which guarantees that $Q \simeq \pi^\ast J \times i_\ast F$ as an extension of $\pi^\ast J$ by $i_\ast F$.
\end{proof}

\subsection{Local finite presentation} \label{sec:lfp}

\begin{lemma} \label{lem:lfp-T}
  Suppose that $A$ is a commutative ring, $J$ is an $A$-module, and the pair $(A, J)$ is the filtered colimit of pairs $(A_i, J_i)$.  Put $V_i = \Spec A_i$ and assume that we are given compatible predeformable families of predeformable maps 
\begin{equation*} \xymatrix{
  U_i \ar[r] \ar[d] & \tsP \ar[d] \\
  V_i \ar[r] & \sS
} \end{equation*}
where $\sS = \tsT$ (resp $\sS = \sT^2$).  Then the natural map
\begin{equation*}
  T(U V, J) = \varinjlim_i T(U_i V_i, J_i) \qquad \qquad \text{(resp. } T'(UV,J) \rightarrow \varinjlim_i T'(U_i V_i, J_i) \text{)}
\end{equation*}
is an equivalence.
\end{lemma}
\begin{proof}
  By the local finite presentation of $\tsP$ over $\sS$ and the local finite presentation of $\sS$, a diagram
  \begin{equation*} \xymatrix{
    U[p^\ast J] \ar[r] \ar[d] & \tsP \ar[d] \\
    V[J] \ar[r] & \sS
  } \end{equation*}
  is induced from a diagram 
  \begin{equation*} \xymatrix{
    U_i[p_i^\ast J_i] \ar[r] \ar[d] & \tsP \ar[d] \\
    V_i[J_i] \ar[r] & \sS .
  } \end{equation*}
  We must check that if the former is predeformable, then so is the latter.  We recall that predeformability is equivalent to the \'etale local existence of a standard form for the diagram:
  \begin{equation*} \xymatrix{
    \Spec A[x,y] / (xy - t) & & \ar[ll] \Spec A[u,v] / (uv - s) \\
    & A \ar[ur] \ar[ul]
  } \end{equation*}
  in which $u \mapsto \alpha x^\ell$, $v \mapsto \beta y^\ell$, where $\alpha$ and $\beta$ are units, and $\alpha \beta \in A$.  But a scheme \'etale over $U$ can be induced from a scheme \'etale over $U_i$ for a sufficiently large $i$, and then the factorization, $\alpha$, $\beta$, and the various equations necessary to demonstrate predeformability will all appear after passing to a large enough $i$.  
\end{proof}

\begin{prop} \label{prop:lfp-E}
  Suppose that $A$ is a commutative ring, $J$ is an $A$-module, and the pair $(A, J)$ is the filtered colimit of pairs $(A_i, J_i)$.  Assume that we are given compatible maps $\Spec A_i \rightarrow \oM_{\rel}(\sP/\BGm)$ (resp.\ $\Spec A_i \rightarrow \oM_{\rel}(\sP)$).  Then $\fE(A, J) = \varinjlim \fE(A_i, J_i)$ (resp.\ $\fE'(A,J) = \varinjlim \fE'(A_i, J_i)$.
\end{prop}
\begin{proof}
  We will show that for $i$ sufficiently large, the maps $\fE(A,J) \rightarrow \fE(A_i, J_i)$ are essentially surjective on objects, essentially surjective on morphisms, surjective on $2$-morphisms, and surjective on equality of $2$-morphisms.  We note first of all that this is a local problem in $S$ since $\fE$ is a stack on $S$.  We can therefore assume that $S$ is affine and in particular quasi-compact and quasi-separated.  Let $C$ be total space of the family of curves over $S$ induced from the map $S \rightarrow \oM_{\rel}(\sP/\BGm)$ (resp.\ $S \rightarrow \oM_{\rel}(\sP)$).  This is also quasi-compact and quasi-separated.

  Any object, morphism, etc.\ in $\fE(S, J)$ can be described by a descent datum on the site $CS$:  if $\xi$ is a $\uT(CS, J)$-torsor (resp.\ a morphism of $\uT(CS, J)$-torsors, resp.\ a pair of $\uT(CS, J)$-torsors) on $CS$ there is a finite diagram of \'etale maps $U_j V_j \rightarrow CS$ such that all of the torsors involved in the definition of $\xi$ become trivial when restricted to each $U_j V_j$, and the collection of $\xi_j$ obtained by restriction to the $U_j V_j$ constitute a descent datum for $\xi$.

  By a ``standard limit argument'' we can assume that the finite diagram of $U_j V_j$ over $CS$ is induced from a diagram of $U_{ij} V_{ij}$ over $C_i S_i$.  Enlarging $i$ as necessary, we can ensure that the maps among the $U_{ij} V_{ij}$ are all \'etale, and we can also ensure the various covering properties of the $U_j V_j$ are inherited from corresponding properties of the $U_{ij} V_{ij}$.  We can certainly descend any torsors involved among the $\xi_j$ to torsors on the $U_{ij} V_{ij}$, since the torsors are trivial on the $U_j V_j$.  Therefore the problem amounts to showing that morphisms between trivial torsors on the $U_j V_j$ are induced from morphisms between torsors on the $U_{ij} V_{ij}$, for $i$ sufficiently large.  But morphisms between trivialized torsors are precisely the same as sections of $\uT(CS, J)$ (resp.\ of $\uT'(CS,J)$), so the problem reduces to exactly what was shown in Lemma~\ref{lem:lfp-T}.
\end{proof}

In \cite{obs}, an obstruction theory satisfying the property demonstrated above is said to be \emph{locally of finite presentation}.

\subsection{Affine pushforward} \label{sec:obs-ax}


We illustrate the functoriality of $\fE(S,J)$ and $\fE'(S,J)$ with respect to affine morphisms in the $S$-variable.  After the discussion in Sections~\ref{sec:obthy1} and~\ref{sec:obthy2}, this is the only remaining axiom of an obstruction theory whose verification is non-trivial.  Suppose that $S \rightarrow S'$ is a morphism over $\oM_{\rel}(\sP/\BGm)$ or over $\oM_{\rel}^\ast(\sP)$.  Then we have a cartesian diagram
\begin{equation*} \xymatrix{
    C \ar[r] \ar[d] & C' \ar[d] \\
    S \ar[r] & S',
  }
\end{equation*}
in which $C$ and $C'$ are the corresponding families over curves over $S$ and $S'$, respectively, and the horizontal arrows are affine.  Let $f : CS \rightarrow C'S'$ be the morphism of sites induced from the diagram above.  Abusively, we also use $f$ to denote the maps $S \rightarrow S'$ and $C \rightarrow C'$.

Let $\uT$ denote the abelian group stack on $CS$ defined as in Section~\ref{sec:obthy1}, and let $\uT_{C'S'}$ denote the corresponding stack on $C'S'$.  Note that because infinitesimal extensions can be pushed out by affine morphisms, and algebraic stacks respect these pushouts, it follows that there is a natural map $f_\ast \uT \rightarrow \uT_{C'S'}$.  Define $\uT'$ and $\uT'_{C'S'}$ likewise.

\begin{lem} \label{lem:push-forward-torsor}
  \begin{enumerate}
  \item     If $P$ is a $T$-torsor on $CS$, then $f_\ast P$ is a $f_\ast T$-torsor on $C'S'$.
  \item  If $P$ is a $T'$-torsor on $CS$ then $f_\ast P$ is a $f_\ast T'$-torsor on $C'S'$.
  \end{enumerate}
\end{lem}

Using the lemma, we can define the maps $\fE(S,J) \rightarrow \fE(S', f_\ast J)$ (resp.\ $\fE'(S,J) \rightarrow \fE(S',f_\ast J)$) by composing the pushforward $P \mapsto f_\ast P$ with the extension of structure group $f_\ast T \rightarrow T_{C'S'}$ (resp.\ $f_\ast T' \rightarrow T'_{C'S'}$).

\begin{proof}[Proof of Lemma \ref{lem:push-forward-torsor}.]
  The proofs of the two statements are similar, so we only prove the first in detail and then explain the modifications necessary for the second.

  Of course, $f_\ast P$ is a $f_\ast \uT$-pseudo-torsor, so the content of the lemma is that $f_\ast P$ admits a section locally in $C'S'$.  Let $L$ be a complex on $CS$ such that $\uT = \ch(L)$.  It will be equivalent to demonstrate that $R^1 f_\ast L = 0$.

  This is a local problem in $C'S'$ so we can assume that $S'$ is the spectrum of a henselian local ring with separably closed residue field.  Localizing in $C'$ as well, we can assume that $C'$ is affine and separate two possibilities depending on whether $C'S'$ is a chart of first or second kind.  In either case, $C$ and $S$ are also affine and $CS$ is a chart of the same kind.

  If $C'S'$ is of first kind, then $\tsP$ is smooth over $\tsT$ (and \'etale over $\sT^2$) near the image of $C'$.  Therefore $\uT(CS,J)$ is an extension of $\pi^\ast \uT_{\tsT}(S,J)$ by $i_\ast \uT_{\tsP / \tsT}(C,J)$ (and $\uT'(CS,J)$ is equal to $\pi^\ast \uT_{\tsT}(S,J)$).  We have
  \begin{gather*}
    H^1(CS, \pi^\ast \uT_{\tsT}(S,J)) = H^1(S, \uT_{\tsT}(S,J)) = 0 \\
    H^1(CS, i_\ast \uT_{\tsP / \tsT}(C,p^\ast J)) = H^1(C, \uT_{\tsP / \tsT}(C,p^\ast J)) = 0
  \end{gather*}
  since both $C$ and $S$ are affine and $\uT_{\tsT}(S,J)$ and both $\uT_{\tsP/\tsT}(C,J)$ are quasi-coherent.

  This leaves the case where $C'S'$ is of second kind to consider.  For simplicity, we will first factor the map $S' \rightarrow \tsT$ through some $W$ that is smooth over $\tsT$ in order to apply Lemma~\ref{lem:A}.

  We note that $R^p f_\ast \pi^\ast \uT_{W/\tsT}(S,J) = 0$ for $p > 0$ since we have
  \begin{equation*}
    H^p(CS, \pi^\ast \uT_{W/\tsT}(S,J)) = H^p(S, \uT_{W/\tsT}(S,J))
  \end{equation*}
  and $\uT_{W/\tsT}(S,J)$ is a quasi-coherent sheaf on the affine scheme $S$.  By Lemma~\ref{lem:A}, this reduces our problem to showing that $R^1 f_\ast \uA(CS, J) = 0$.

  We recall from Section~\ref{sec:calc} that $\uA(CS,J)$ can be represented as the sheaf of diagrams~\eqref{eqn:20}.  Note that $\Omega^\dagger_{\sP_W}$ and $\Omega_W^\dagger$ actually are sheaves (on $C$ and $S$, respectively) since we have restricted to the case where $C'S'$ (and therefore also $CS$) is of second kind.  We can now filter $\uA(CS,J)$ as
  \begin{equation*}
    0 \rightarrow i_\ast \uHom(\Omega_{\sP_W / W}(\log), p^\ast J) \rightarrow \uA(CS,J) \rightarrow \pi^\ast \uHom(\Omega_W^\dagger, J) \rightarrow 0 .
  \end{equation*}
  Note, however, that
  \begin{gather*}
    H^1(CS, i_\ast \uHom(\Omega_{\sP_W/W}(\log), p^\ast J)) = H^1(C, \uHom(\Omega_{\sP_W/W}(\log), p^\ast J)) = 0 \\
    H^1(CS, \pi^\ast \uHom(\Omega_W^\dagger, J)) = H^1(S, \uHom(\Omega_W^\dagger, J)) = 0
  \end{gather*}
  since both $C$ and $S$ are affine, and both $\uHom(\Omega_{\sP_W/W}(\log), p^\ast J)$ and $\uHom(\Omega_W^\dagger, J)$ are quasi-coherent.  Therefore $H^1(CS, \uA(CS, J)) = 0$ and we are done.
\end{proof}

%
\subsection{Perfection} \label{sec:perfection}

\begin{prop} \label{prop:perfect}
  \begin{enumerate}
  \item $\fE$ is a perfect relative obstruction theory for $\oM_{\rel}(\sP/\BGm)$.
  \item $\fE'$ is a perfect relative obstruction theory for $\fM_{\rel}^\ast(\sP)$ over $\fM$.
  \end{enumerate}
\end{prop}
\begin{proof}
  In view of~\cite[Compl\'ement~I.4.11]{sga6} and the exact sequence~\eqref{eqn:16}, either of these assertions implies the other, since the isomorphisms $\fE''(S,J) \simeq \uHom(\bE, J)$ implies $\fE''$ is represented by a vector bundle.  We prove the first.

  To show that $\fE$ is a vector bundle stack, we must show that its restriction to any $S$-point of $\oM_{\rel}(\sP/\BGm)$ is a vector bundle stack.  This problem is local in $S$, so we are free to assume that the map $S \rightarrow \tsT$ factors through a smooth map $W \rightarrow \tsT$.  Let $\tsP_W = \tsP \fp_{\tsT} W$ be the expansion of $\sP$ over $W$ induced from the map from $W$ to $\tsT$.  

  Pushing forward the exact sequence~\eqref{eqn:23} from Lemma~\ref{lem:A} and using the fact that $R \pi_\ast \pi^\ast = \id$, we get an exact sequence
  \begin{equation*}
    0 \rightarrow f^\ast T_{W/\tsT} \rightarrow \fE \rightarrow \fF \rightarrow 0
  \end{equation*}
  where $\fF$ denotes the category of torsors under $\uA(CS,J)$ on $CS$.  Since $f^\ast T_{W/\tsT}$ is a vector bundle, $W$ being smooth and representable over $\tsT$, the proposition reduces by \cite[Compl\'ement~I.4.11]{sga6} to showing that $\fF$ is perfect.

  Now use the exact sequence~\eqref{eqn:21} from Lemma~\ref{lem:B} to get an exact sequence
  \begin{equation*}
    0 \rightarrow \fG' \rightarrow \fF \rightarrow \fG \rightarrow 0
  \end{equation*}
  where $\fG'$ is the stack of $i_\ast \uHom(g^\ast \Omega_{\tsP/\tsT}(\log), p^\ast J)$-torsors on $CS$ and $\fG$ is the stack of $\uB(CS,J)$-torsors on $CS$.  Note that the map $\fF \rightarrow \fG$ is surjective because
  \begin{equation*}
    H^2(CS, i_\ast \uHom(g^\ast \Omega_{\tsP/\tsT}(\log), p^\ast J)) = H^2(C, g^\ast \Omega_{\tsP/\tsT}(\log), p^\ast J) = 0
  \end{equation*}
  because $C$ is a curve.  Since $g^\ast \Omega_{\tsP/\tsT}(\log)$ is a vector bundle, $\fG'(S, \cO_S)$ is representable by $R p_\ast g^\ast \Omega_{\tsP/\tsT}(\log)^\vee[1]$.  In particular, this is a perfect complex of perfect amplitude in $[-1,0]$.  To show $\fF$ is perfect, it therefore suffices (again by \cite[Compl\'ement~I.4.11]{sga6}) to see that $\fG$ is perfect.

  By Lemma~\ref{lem:B}, $\uB(CS,J)$ is representable by $\uHom(\Omega_W^\dagger, J)$.  Let $L$ be the sheaf $\uB(CS, \cO_S)$ on $CS$.  There is an injective map $\pi^\ast f^\ast \Omega_W \rightarrow \Omega_W^\dagger$ which is an isomorphism on charts of first kind.  Let $K$ be the cone of the dual map $L \rightarrow \pi^\ast f^\ast \Omega_W^\vee$.  This gives an exact triangle
  \begin{equation*}
    R \pi_\ast L \rightarrow f^\ast \Omega_W^\vee \rightarrow R \pi_\ast K \rightarrow R \pi_\ast L[1] .
  \end{equation*}
  To show that $L[1]$ is perfect in degrees $[-1,0]$ it suffices (by \cite[Compl\'ement~I.4.11]{sga6}) to see that $R \pi_\ast K$ is perfect in $[-1,0]$.

  Now, $K$ is zero on charts of first kind, so if $h : D \rightarrow CS$ is the closed embedding induced from the inclusion $D \subset C$ of the essential nodes of $C$, then $K = h_\ast h^\ast K$.  The functor $h_\ast$ is exact since $h$ is a closed embedding, so $R \pi_\ast K = R \psi_\ast h^\ast K$, where $\psi : D \rightarrow S$ is the projection.  But $D$ is finite over $S$ so $\psi_\ast$ is exact and $R \psi_\ast h^\ast K = \psi_\ast h^\ast K$.  

  Note now that $h^\ast K$ is a $2$-term complex of locally free $\psi^\ast \cO_S$-modules, concentrated in degrees $[-1,0]$, so $\psi_\ast h^\ast K$ is as well.  Indeed, we can assume after \'etale localization in $S$ that $D$ can be split into a disjoint union of components on which the map to $S$ is a closed embedding.  The question reduces to the consideration of a single component, so we can assume that $D$ is closed in $S$.  Then $\psi_\ast h^\ast K$ is representable near $D$ by the complex
  \begin{equation*}
    [ \Hom(\Omega_W^\dagger, \cO_S) \rightarrow \Hom(g^\ast \Omega_W, \cO_S) ].
  \end{equation*}



\end{proof}

\bibliographystyle{amsalpha}             
\bibliography{rubber}       
\end{document}